\documentclass[11pt]{amsart}

\usepackage[margin=1in]{geometry}

\usepackage{amssymb,amsmath,bbm}
\usepackage{graphicx}
\usepackage[backref=page]{hyperref}
\usepackage[width=12cm,format=plain,labelfont=sf,bf,small,textfont=sf,up,small,justification=justified,labelsep=period]{caption}
\usepackage{subcaption}

\hypersetup{
    colorlinks=true,       
    linkcolor=blue,          
    citecolor=blue,        
    filecolor=blue,      
    urlcolor=blue           
}

\makeatletter
 \renewcommand*\l@subsection{\@tocline{2}{0em}{2.8em}{6.4em}{}}
\makeatother

\newcommand{\RR}{\mathbb{R}}
\newcommand{\QQ}{\mathbb{Q}}
\newcommand{\CC}{\mathbb{C}}
\newcommand{\XX}{\mathcal{X}}
\newcommand{\PSD}{\mathcal{S}_+}
\newcommand{\rank}{\textup{rank}\,}
\newcommand{\rankplus}{\textup{rank}_+\,}
\newcommand{\rankpsd}{\textup{rank}_{\textup{psd}}\,}
\newcommand{\rankpsdQ}{\textup{rank}_{\textup{psd}}^{\QQ}\,}
\newcommand{\rankpsdC}{\textup{rank}_{\textup{psd}}^{\CC}\,}
\newcommand{\sqrtrank}{\textup{rank}_{\! \! {\sqrt{\ }}}\,}
\newcommand{\diag}{\textup{diag}}
\newcommand{\supp}{\textup{supp}}
\newcommand{\conv}{\textup{conv}}
\newcommand{\cone}{\textup{cone}}

\newcommand{\onevec}{\mathbbm{1}}
\newcommand{\NN}{\mathbb{N}}

\newcommand{\CalS}{\mathcal{S}}
\newcommand{\SF}{\mathcal{S}\mathcal{F}}
\renewcommand{\S}{\mathcal S} 

\DeclareMathOperator{\trace}{trace}
\DeclareMathOperator{\range}{range}
\newcommand{\cL}{\mathcal L}

\newcommand{\hadprod}{\circ} 
\newcommand{\CS}{\text{CP}_{\text{psd}}} 
\newcommand{\CP}{\text{CP}} 
\newcommand{\DN}{\text{DN}} 
\let\Re\relax
\let\Im\relax
\DeclareMathOperator{\Re}{Re}
\DeclareMathOperator{\Im}{Im}
\DeclareMathOperator{\vecm}{vec}

\newtheorem{theorem}{Theorem}[section]

\newtheorem{lemma}[theorem]{Lemma}
\newtheorem{corollary}[theorem]{Corollary}
\newtheorem{proposition}[theorem]{Proposition}
\theoremstyle{definition}
\newtheorem{definition}[theorem]{Definition}
\newtheorem{example}[theorem]{Example}

\theoremstyle{remark}
\newtheorem{remark}[theorem]{Remark}
\newtheorem{problem}[theorem]{Problem}

\title{Positive Semidefinite Rank}

\author[H. Fawzi]{Hamza Fawzi} \address{Laboratory for Information and
  Decision Systems (LIDS), Massachusetts Institute of Technology,
  Cambridge, MA 02139, USA} \email{hfawzi@mit.edu}

\author[J. Gouveia]{Jo{\~a}o Gouveia}
\address{CMUC, Department of Mathematics,
  University of Coimbra, 3001-454 Coimbra, Portugal}
\email{jgouveia@mat.uc.pt}

\author[P.A. Parrilo]{Pablo A. Parrilo}\address{Laboratory for
  Information and Decision Systems (LIDS), Massachusetts Institute of
  Technology, Cambridge, MA 02139, USA} \email{parrilo@mit.edu}

\author[R.Z. Robinson]{Richard Z. Robinson}
\address{Department of Mathematics, University of Washington, Box
  354350, Seattle, WA 98195, USA} \email{rzr@uw.edu}

\author[R.R. Thomas]{Rekha R. Thomas}
\address{Department of Mathematics, University of Washington, Box
  354350, Seattle, WA 98195, USA} \email{rrthomas@uw.edu}

\thanks{Fawzi and Parrilo were supported in part by AFOSR
  FA9550-11-1-0305.  Gouveia was supported by the Centre for
  Mathematics at the University of Coimbra and Funda\c{c}\~ao para a
  Ci\^encia e a Tecnologia, through the European program
  COMPETE/FEDER. Robinson was supported by the US NSF Graduate Research Fellowship through grant DGE-1256082 and Thomas was supported by NSF grant DMS-1115293. }

\begin{document}

\begin{abstract}
Let $M \in \RR^{p \times q}$ be a nonnegative matrix. The positive
semidefinite rank (psd rank) of $M$ is the smallest integer $k$ for
which there exist positive semidefinite matrices $A_i, B_j$ of size $k
\times k$ such that $M_{ij} = \trace(A_i B_j)$. The psd rank has many
appealing geometric interpretations, including semidefinite
representations of polyhedra and information-theoretic
applications. In this paper we develop and survey the main
mathematical properties of psd rank, including its geometry,
relationships with other rank notions, and computational and
algorithmic aspects.
\end{abstract}

\maketitle

\tableofcontents

\section{Introduction}

Matrix factorizations (or more generally, factorizations of linear
maps) are a classical and important topic in applied mathematics. For
instance, in the standard low-rank matrix factorization problem, given
a matrix $M \in \RR^{p \times q}$ one constructs matrices $A \in
\RR^{p \times k}$ and $B \in \RR^{k \times q}$ such that $M = A B$,
where the intermediate dimension $k$ is as small as possible.  Letting
$a_i$ and $b_j$ be the rows of $A$ and columns of $B$, respectively,
finding such a factorization can be interpreted as
a \emph{realizability} problem, where we want to produce vectors in
$\RR^k$ that realize the inner products given by $M_{ij} = \langle
a_i, b_j \rangle$. The smallest such $k$ is, of course, the usual rank
of the matrix $M$.

In applications, low-rank factorizations often have appealing
interpretations (e.g., reduced-order or ``simple'' models), since they
provide a decomposition of a linear map $\RR^q \rightarrow \RR^p$ in
terms of mappings $\RR^q \rightarrow \RR^k \rightarrow \RR^p$ through
a ``small'' subspace.  Many classical and successful methods in
systems theory (e.g., realization theory, model order reduction),
or statistics and machine learning (e.g., principal component
analysis, factor analysis) are based on these techniques; see
e.g. \cite{kalman1963mathematical,MoorePCA,JolliffePCA}.

In many situations, however, one often requires additional conditions
on the possible factors. A well-known example is the case
of \emph{nonnegative factorizations} \cite{cohen1993nonnegative},
where $M$ is a given nonnegative matrix and the factors $A,B$ are also
required to be nonnegative (here and throughout the paper, a \emph{nonnegative} matrix is a matrix where all the entries are nonnegative). These requirements often arise from
probabilistic interpretations (if $M$ corresponds to a joint
distribution, in which case the factors can be interpreted in terms of
conditional independence; see e.g. \cite{straten2003sandwich}), or
modeling choices (additive representations in terms of features and
latent variables; see e.g., \cite{LeeSeung}). Another well-known case
is when the factors $A,B$ are required to be ``small'' with respect to
a given matrix norm. This is a situation that has been well-studied in
contexts such as Banach space theory, communication complexity and
machine learning, where \emph{factorization norms} are used to capture
this notion; see
e.g. \cite{LinialMendelsonSchechtmanShraibman,LinialShraibman}.

Over the last couple of years, a new and intriguing class of matrix
factorization problems has been introduced, by considering
\emph{conic factorizations} of nonnegative linear maps through a \emph{convex cone} $K$, i.e., 
mappings $\RR^q_+ \rightarrow K \rightarrow \RR^p_+$ (nonnegative factorizations
correspond to the case when $K$ is the nonnegative orthant). A
particularly interesting case, which is the focus of this paper, occurs
when $K$ is the cone of positive semidefinite matrices
\cite{gouveia2011lifts}. 

More concretely, a \emph{positive semidefinite factorization} of a
nonnegative matrix $M \in \RR^{p \times q}$ is a collection of
symmetric $k \times k$ positive semidefinite matrices $\{A_1,\ldots,A_p\}$ and
$\{B_1,\ldots,B_q\}$ such that
\[
M_{ij} = \trace (A_i B_j), \qquad i=1,\ldots,p, \quad j=1,\ldots,q.
\]
The \emph{positive semidefinite rank} (\emph{psd rank}) of $M$ can
then be defined as the smallest $k$ for which such a factorization
exists \cite{gouveia2011lifts}.  As we explain in Section~\ref{sec:motivation}, a
natural and important source of these factorizations is their
relationship with representability of polytopes by semidefinite
programming. These results extend the connections, first explored by
Yannakakis in the context of polytopes and linear
programming \cite{Yannakakis}, between ``algebraic'' factorizations of
the slack matrix and the ``geometric'' question of existence of
extended formulations.  Since this quantity exactly characterizes
semidefinite representability, positive semidefinite rank is an
essential component of the burgeoning area of convex algebraic
geometry \cite{SDOCAG}.

Besides these geometric and complexity-theoretic considerations,
however, there are many other reasons to study these natural
factorizations and ranks as independent mathematical objects, and this
is the viewpoint we emphasize in this paper.  Our main goal is to
study the positive semidefinite rank of a matrix from the
algebraic-geometric perspective, as well as survey and collect most of
the existing results to date.

\subsection{Paper outline}
In Section~\ref{sec:definitions} we present the formal definition and
basic properties of psd rank. We analyze its behavior under natural
matrix operations, its continuity properties, as well as its
dependence on the underlying field. Throughout, we present numerous
examples illustrating these notions.

Section~\ref{sec:motivation} presents several complementary
interpretations and motivations for psd rank. We discuss extensively
its main geometric interpretation in terms of the ``complexity'' of a
convex body that is contained in between two polytopes, a topic
initially studied in \cite{gouveia2011lifts} and further elaborated
in \cite{fiorini2012lowerbound}. We also discuss a quantum analogue of
the well-known probabilistic interpretations of nonnegative
factorizations in terms of conditional independence, where the psd
rank characterizes the minimum amount of quantum information that must
be shared between two parties to generate samples of a correlated
random variable \cite{jain2013efficient}.

Like nonnegative rank, psd rank can be computationally challenging,
although in some situations it can be nicely characterized. In
Section~\ref{sec:psdrank2_convexprogramming} we show that the case of
psd rank equal to 2 can be decided using convex optimization (in
particular, semidefinite programming). The positive semidefinite rank
of a matrix has natural relations with other rank notions, such as the
usual (or ``standard'') rank and the nonnegative rank; we discuss
these in Section~\ref{sec:relationships_between_ranks}. We also
present the related notion of ``square root rank,'' a refinement of
psd rank to the case of rank-one factors. We show that in general,
these rank notions are incomparable (Table~\ref{table:ranks}), and
provide explicit examples/counterexamples for all pairwise comparisons
between them.

In Section~\ref{sec:properties} we analyze the situation where
additional properties are imposed on the factor matrices $A_i,
B_j$. We show how to guarantee uniform bounds on the factors for
different norms (trace, spectral norm), as well as upper bounds on
their ranks.

Since psd factorizations are not unique, it is also of interest to
study the space of all possible factorizations. In
Section~\ref{sec:space of factorizations} we study the topological
properties of the space of factorizations, and show that in certain
cases it is connected, in contrast to the case of nonnegative
factorizations. This geometric insight also allows a better
understanding on the rank of possible factors; see e.g.,
Example~\ref{ex:full rank psd factors}.

In Section~\ref{sec:symmetric_psd_factorizations} we specialize the
situation to symmetric matrices with symmetric factorizations, and
discuss the connections with some classical matrix cones (completely
positive, doubly nonnegative).


We conclude in Section~\ref{sec:open} with a list of open problems,
and questions for future research.


\section{Definition and interpretations}
\label{sec:definitions}

For a positive integer $k$, let $\S^k_+$ denote the cone of $k \times k$ real symmetric positive semidefinite (psd) matrices. Recall that $\S^k_+$ is a closed convex cone in the vector space $\S^k$ of all $k \times k$ real symmetric matrices. We equip $\S^k$ with the standard inner product defined by:
\[ \langle A, B \rangle = \trace(AB) = \sum_{1\leq i,j \leq k} A_{ij} B_{ij}. \]
The inner product of any two psd matrices is nonnegative, i.e., if $A,B \in \S^k_+$, then $\langle A,B \rangle \geq 0$. In fact the cone $\S^k_+$ is \emph{self-dual}, meaning that:
\[ A \in \S^k_+ \; \Longleftrightarrow \; \langle A, X \rangle \geq 0 \quad \forall X \in \S^k_+. \]
The following well-known fact about orthogonal matrices in $\S^k_+$ will be useful later:
\begin{proposition}
\label{prop:orthpsd}
If $A, B \in \S^k_+$ are such that $\langle A, B \rangle = 0$, then $AB=0$.
\end{proposition}
\begin{proof}
If we let $A = \sum_i a_i a_i^T$ and $B = \sum_j b_j b_j^T$ then $\langle A, B \rangle = \sum_{i,j} (a_i^T b_j)^2$. Thus since $\langle A, B \rangle = 0$ we get $a_i^T b_j = 0$ for all $i,j$. Hence this means that $AB = \sum_{i,j} (a_i^T b_j) a_i b_j^T = 0$.
\end{proof}


\subsection{Psd rank}
We now give the formal definitions of \emph{psd factorizations} and \emph{psd rank}, which are the main objects of study in this paper:
\begin{definition}[\cite{gouveia2011lifts}]
\label{def:psdfactorization}
Given a nonnegative matrix $M \in \RR^{p\times q}_+$, a \emph{psd
factorization} of $M$ of size $k$ is a collection of psd matrices
$A_1,\dots,A_p \in \S^k_+$ and $B_1,\dots,B_q \in \S^k_+$ such that
$M_{ij} = \langle A_i, B_j \rangle$ for all $i=1,\dots,p$ and
$j=1,\dots,q$. The \emph{psd rank} of $M$, denoted $\rankpsd(M)$, is the smallest
integer~$k$ for which $M$ admits a psd factorization of size $k$.
\end{definition}

\begin{remark}
A psd factorization of $M$ is equivalent to the existence of linear
maps $\RR^q_+ \rightarrow \S^k_+ \rightarrow \RR^p_+$. Indeed, given a
psd factorization, the linear maps $x \mapsto \sum_{j=1}^q x_j B_j$
and $Y \mapsto \langle A_i, Y \rangle$ (for $i =1,\ldots,p$) have the
desired property. The converse is also easy, by considering the image
of the coordinate vectors $e_1, \ldots, e_q$ under the first map, and
self-duality of the cone $\S^k_+$.
\end{remark}

The psd rank is related to another notion of rank for nonnegative matrices, namely the \emph{nonnegative rank} \cite{cohen1993nonnegative}.
\begin{definition}
Given a nonnegative matrix $M \in \RR^{p\times q}_+$, a \emph{nonnegative factorization} of $M$ of size $k$ is a collection of nonnegative vectors $a_1,\dots,a_p \in \RR^k_+$ and $b_1,\dots,b_q \in \RR^k_+$ such that $M_{ij} = a_i^T b_j$ for all $i=1,\dots,p$ and $j=1,\dots,q$.\\
The \emph{nonnegative rank} of $M$, denoted $\rank_+(M)$, is the smallest integer $k$ for which $M$ admits a nonnegative factorization of size $k$.
\end{definition}
The first proposition below establishes simple inequalities between the different notions of rank, namely the usual rank, the psd rank and the nonnegative rank.
\begin{proposition}
\label{prop:ineqsranks}
If $M \in \RR^{p\times q}_+$ is a nonnegative matrix, then
\begin{equation}
\label{eq:ineqsranks}
 \frac{1}{2}\sqrt{1+8 \,\rank(M)}-\frac{1}{2} \leq \rankpsd(M) \leq \rank_+(M) \leq \min(p,q).
\end{equation}
\end{proposition}
\begin{proof}
The last inequality is trivially true since $M = MI_q = I_pM$ where $I_k$ is the $k \times k$ identity matrix.

Suppose $a_1,\ldots,a_p \in \RR^k_+$ and $b_1,\ldots,b_q \in \RR^k_+$ give a nonnegative factorization of $M \in \RR^{p \times q}_+$. Then the diagonal matrices $A_i:= \diag(a_i)$ and
$B_j :=\diag(b_j)$ give a $\S^k_+$-factorization of $M$, and we obtain the second inequality.

Now suppose $A_1, \ldots A_p, B_1, \ldots, B_q $ give a $\S^k_+$-factorization of $M$.
Consider the vectors
\[ a_i = \vecm(A_i) \quad \text{ and } \quad b_j = \vecm(B_j) \]
where for $X \in \S^k$ we define $\vecm(X) \in \RR^{{k+1 \choose 2}}$ by:
\[ \vecm(X) = (X_{11},\ldots,X_{kk},\sqrt{2}X_{12},\ldots,\sqrt{2}X_{1r},\sqrt{2}X_{23}, \ldots, \sqrt{2} X_{(k-1)k}). \]
Then $\langle a_i,b_j \rangle =  \langle A_i, B_j \rangle= M_{ij}$ so $M$ has rank at most ${k+1 \choose 2}$. By solving for $k$ we get the desired inequality.
%
\end{proof}

\begin{example}
\label{ex:3x3derangement}
To illustrate the notion of a psd factorization, consider the following matrix $M$ known as the $3\times 3$ \emph{derangement} matrix:
\[ M = \begin{bmatrix} 0 & 1 & 1\\ 1 & 0 & 1\\ 1 & 1 & 0\end{bmatrix}. \]
This matrix $M$ satisfies $\rank(M) = \rank_+(M) = 3$. One can show that $M$ admits a psd factorization of size 2. Indeed, define:
\[
\begin{array}{lll}
A_1 = \left[ \begin{array}{cc} 1 & 0 \\ 0 & 0 \end{array} \right] &
A_2 = \left[ \begin{array}{cc} 0 & 0 \\ 0 & 1 \end{array} \right] &
A_3 = \left[ \begin{array}{cc} 1 & -1 \\ -1 & 1 \end{array} \right]\\[0.5cm]
B_1 = \left[ \begin{array}{cc} 0 & 0 \\ 0 & 1 \end{array} \right] &
B_2 = \left[ \begin{array}{cc} 1 & 0 \\ 0 & 0 \end{array} \right] &
B_3 =  \left[ \begin{array}{cc} 1 & 1 \\ 1 & 1 \end{array} \right].
\end{array}
\]
One can easily check that the matrices $A_i$ and $B_j$ are positive semidefinite, and that $M_{ij} = \langle A_i, B_j \rangle$ for all $i=1,\dots,3$ and $j=1,\dots,3$. This factorization shows that $\rankpsd(M) \leq 2$. In fact one has $\rankpsd(M) = 2$ since the first inequality in \eqref{eq:ineqsranks} gives $\rankpsd(M) \geq \frac{1}{2}\sqrt{1+8 \cdot 3}-\frac{1}{2} = 2$. $\lozenge$
\end{example}

\begin{example}
\label{ex:3x3circulant}
Consider more generally the following $3\times 3$ \emph{circulant} matrix, where $a,b,c$ are nonnegative real numbers:
\begin{equation}
\label{eq:3x3circulant}
 M(a,b,c) = \begin{bmatrix} a & b & c\\ c & a & b\\ b & c & a\end{bmatrix}.
\end{equation}
One can check that the usual rank of $M(a,b,c)$ is 3 unless $a=b=c$ in which case the rank is one. When $\rank(M(a,b,c)) = 3$, the bounds in \eqref{eq:ineqsranks} say that $2 \leq \rankpsd(M(a,b,c)) \leq 3$. Using the geometric interpretation of the psd rank (cf. Section \ref{sec:motivation}) one can show that
\[ \rankpsd(M(a,b,c)) \leq 2 \quad \Longleftrightarrow \quad a^2+b^2+c^2 \leq 2(ab+bc+ac). \]
Figure \ref{fig:3x3circulant} shows the region described by the inequality above when $a=1$.
\begin{figure}[ht]
  \centering
  \includegraphics[width=8cm]{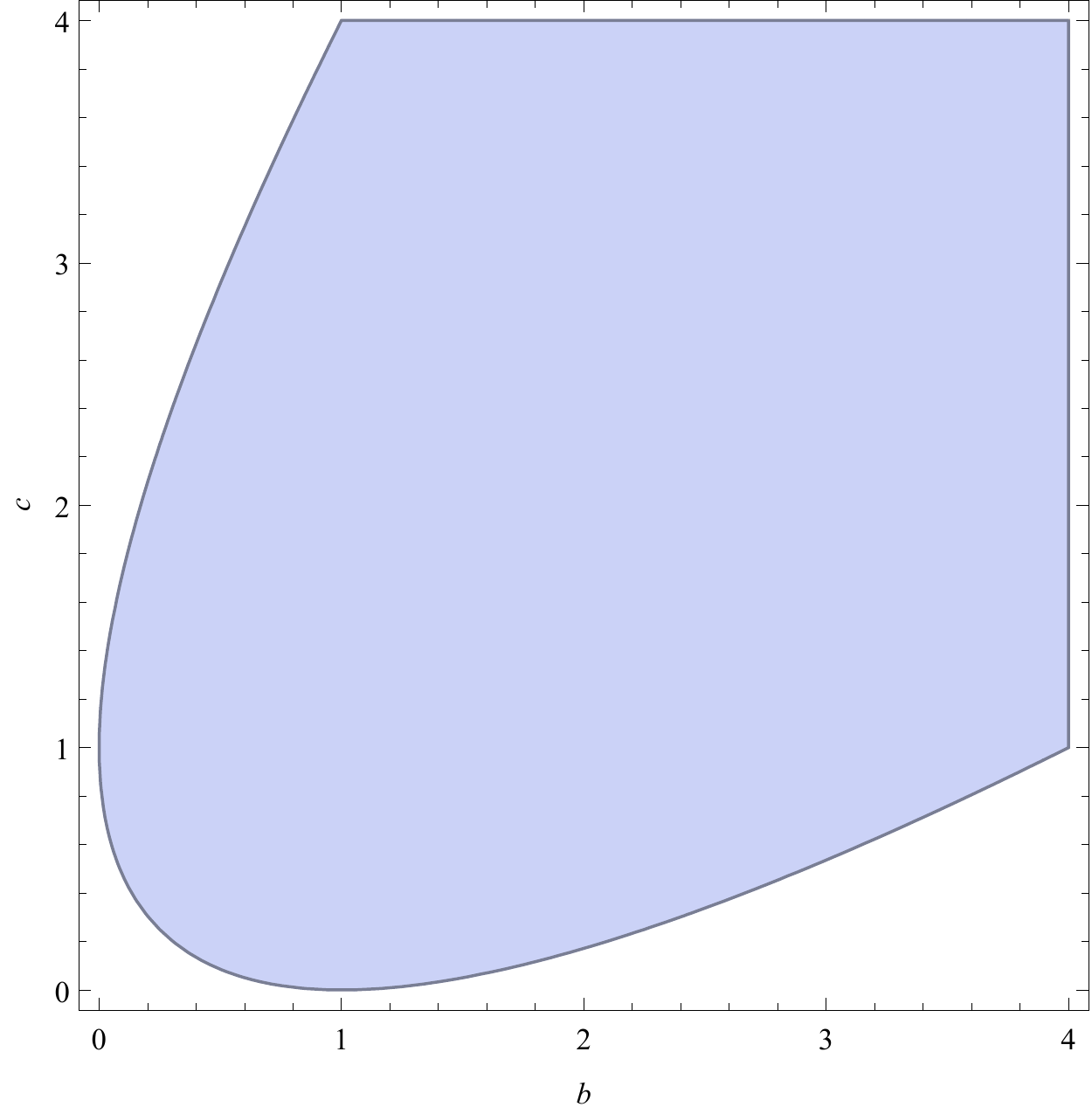}
  \caption{The blue region shows the values of $(b,c)$ for which $\rankpsd(M(1,b,c)) \leq 2$, where $M(a,b,c)$ is the $3\times 3$ circulant matrix defined in \eqref{eq:3x3circulant}. This region is defined by the inequality $1+b^2+c^2 \leq 2(b+c+bc)$.}
  \label{fig:3x3circulant}
\end{figure}
$\lozenge$
\end{example}

If $M$ is a matrix such that $\rank(M) = 1$ or $\rank(M) = 2$, then the psd rank is equal to the rank, as stated in the next proposition:
\begin{proposition}
\label{prop:rank1rank2}
For a nonnegative matrix $M \in \RR^{p\times q}_+$ the following is true:
\begin{equation}
 \label{eq:equivrankone}
 \rank(M) = 1 \; \Leftrightarrow \; \rank_+(M) = 1 \; \Leftrightarrow \; \rankpsd(M) = 1.
\end{equation}
Furthermore, we have the following implication:
\begin{equation}
\label{eq:implicationsranktwo}
\rank(M) = 2 \;\; \Rightarrow \;\; \rank_+(M) = \rankpsd(M) = 2.
\end{equation}
\end{proposition}
\begin{proof}
The proof of \eqref{eq:equivrankone} is clear from the inequalities \eqref{eq:ineqsranks}.
For \eqref{eq:implicationsranktwo}, we can use a result from \cite{cohen1993nonnegative} which states that if $\rank(M) = 2$ then $\rank_+(M) = 2$, from which it easily follows that $\rankpsd(M) = 2$. In Section \ref{sec:motivation} (Remark \ref{rem:slackmatrices}), we give a geometric argument for \eqref{eq:implicationsranktwo}.
\end{proof}

\subsection{Dependence on the field} In the definition of psd rank, Definition \ref{def:psdfactorization}, we required the factors $(A_i)_{i=1,\dots,p}$ and $(B_j)_{j=1,\dots,q}$ to have real entries. When the matrix $M$ has rational entries, it is natural to define a notion of psd rank where the factors $(A_i)_{i=1,\dots,p}$ and $(B_j)_{j=1,\dots,q}$ are required to have rational entries. If we denote this by $\rankpsdQ(M)$, then we clearly have:
\begin{equation}
\label{eq:ineqpsdrankQ}
 \rankpsd(M) \leq \rankpsdQ(M).
\end{equation}
In \cite{gouveia2014rational} it was shown on a $8\times 6$ matrix $M$ that the inequality \eqref{eq:ineqpsdrankQ} can be strict.

It is also natural to consider a related notion of psd rank where the
factors $(A_i)_{i=1,\dots,p}$ and $(B_j)_{j=1,\dots,q}$ in the psd
factorizations are positive semidefinite \emph{Hermitian}
matrices. Denote by $\rankpsdC(M)$ the associated psd rank. It is not
difficult to see that the following inequalities hold:
\[
\rankpsdC(M) \leq \rankpsd(M) \leq 2 \, \rankpsdC(M).
\]
The second inequality comes from the fact that if $A$ is a Hermitian positive semidefinite matrix, then the $2n\times 2n$ real symmetric matrix
\begin{equation}
 \label{eq:hermitianembed}
 \frac{1}{\sqrt{2}} \begin{bmatrix} \Re A & \Im A\\ -\Im A & \Re A \end{bmatrix}
\end{equation}
is positive semidefinite. Furthermore the function which maps any $n\times n$ Hermitian matrix $A$ to the $2n\times 2n$ block matrix \eqref{eq:hermitianembed} preserves inner products.

One can show that the Hermitian psd rank can be strictly smaller than the real psd rank. Consider the $4\times 4$ derangement matrix:
\[
M = 
\begin{bmatrix}
0 & 1 & 1 & 1\\
1 & 0 & 1 & 1\\
1 & 1 & 0 & 1\\
1 & 1 & 1 & 0
\end{bmatrix}.
\]
Using the inequalities \eqref{eq:ineqsranks} one can show that $\rankpsd(M) \geq 3$. However one can find a psd factorization of $M$ with Hermitian matrices of size 2, as given below:
\[
\footnotesize
\begin{array}{llll}
A_1 = \left[ \begin{array}{cc} 1 & 0 \\ 0 & 0 \end{array} \right] &
A_2 = \left[ \begin{array}{cc} 0 & 0 \\ 0 & 1 \end{array} \right] &
A_3 = \left[ \begin{array}{cc} 1 & -1 \\ -1 & 1 \end{array} \right] &
A_4 = \left[ \begin{array}{cc} 1 & e^{2i\pi/3} \\ e^{-2i\pi/3} & 1 \end{array} \right]
\\[0.5cm]
B_1 = \left[ \begin{array}{cc} 0 & 0 \\ 0 & 1 \end{array} \right] & 
B_2 = \left[ \begin{array}{cc} 1 & 0 \\ 0 & 0 \end{array} \right] & 
B_3 =  \left[ \begin{array}{cc} 1 & 1 \\ 1 & 1 \end{array} \right] & 
B_4 = \left[ \begin{array}{cc} 1 & -e^{2i\pi/3}\\ -e^{-2i\pi/3} & 1 \end{array} \right].
\end{array}
\]
Actually in \cite{TroyLeeHermitianPsdRank}, the authors exhibit a sequence of matrices $(M_k)$ of increasing size such that $\rankpsdC(M_k) < \rankpsd(M_k)$ for all $k$ and where
the gap $\rankpsd(M_k) - \rankpsdC(M_k)$ grows with $k$ (the ratio $\rankpsd(M_k) / \rankpsdC(M_k)$ goes asymptotically to $\sqrt{2}$).

In this survey we will focus on the \emph{real} psd rank, given in Definition~\ref{def:psdfactorization}.

\subsection{First properties} The next theorem establishes some structural properties satisfied by the psd rank
\begin{theorem}
\label{thm:firstprops}
Given a nonnegative matrix $M \in \RR^{p\times q}_+$, we have:
\begin{itemize}
\item[(i)] $\rankpsd(M) = \rankpsd(M^T)$.
\item[(ii)] If $D_1 \in \RR^{p\times p}_+, D_2\in \RR^{q\times q}_+$ are diagonal matrices with strictly positive elements on the diagonal, then
$\rankpsd(D_1 M D_2) = \rankpsd(M)$.
\item[(iii)] If $N \in \RR^{p\times q}_+$, then $\rankpsd(M+N) \leq \rankpsd(M) + \rankpsd(N)$.
\item[(iv)] If $N \in \RR^{q\times r}_+$ then $\rankpsd(MN) \leq \min(\rankpsd(M), \rankpsd(N))$.
\item[(v)] $\rankpsd(M\hadprod M) \leq \rank(M)$, where $\hadprod$ denotes Hadamard (entrywise) product.
\end{itemize}
\end{theorem}
\begin{proof}
~\begin{itemize}
\item[(i)] Property (i) is clear.
\item[(ii)] If $M_{ij} = \langle A_i, B_j \rangle$ is a psd factorization of $M$, then
\[ (D_1 M D_2)_{ij} = \langle (D_1)_{ii} A_i, (D_2)_{jj} B_j \rangle \] is a psd factorization of $D_1 M D_2$ of the same size. Thus since the diagonal elements of $D_1$ and $D_2$ are strictly positive we easily get that $\rankpsd(M) = \rankpsd(D_1 M D_2)$.
\item[(iii)] Let $M_{ij} = \langle A_i, B_j \rangle$ and $N_{ij} = \langle A'_i, B'_j \rangle$ be psd factorizations of $M$ and $N$ of size respectively $\rankpsd(M)$ and $\rankpsd(N)$. Define
\[ C_i = \begin{bmatrix} A_i & 0\\ 0 & A'_i \end{bmatrix} \quad \text{ and } \quad D_j = \begin{bmatrix} B_j & 0\\ 0 & B'_j \end{bmatrix}. \]
Note that $C_i$ and $D_j$ are psd matrices of size $\rankpsd(M)+\rankpsd(N)$. Furthermore we clearly have $M_{ij} + N_{ij} = \langle C_i, D_j \rangle$. Thus $\rankpsd(M+N) \leq \rankpsd(M)+\rankpsd(N)$.
\item[(iv)] Let $k=\rankpsd(M)$ and let $M_{ij} = \langle A_i, B_j \rangle$ be a psd factorization of $M$ of size $k$.
 For $j \in [r]$, define $C_j = \sum_{t=1}^{q} N_{tj} B_t$. Note that $C_j \in \S^k_+$ since $N_{tj} \geq 0$ and $B_t \in \S^k_+$. Then we verify that $\langle A_i, C_j \rangle = \sum_{t=1}^q N_{tj} \langle A_i, B_t \rangle = \sum_{t=1}^q N_{tj} M_{it} = (MN)_{ij}$ and so we get a psd factorization of $MN$ of size $k$. This shows that $\rankpsd(MN) \leq \rankpsd(M)$. A similar argument shows that $\rankpsd(MN) \leq \rankpsd(N)$.
\item[(v)] Let $M_{ij} = \langle a_i, b_j \rangle$ be a factorization of $M$ where $a_i, b_j \in \RR^r$ where $r = \rank(M)$. Define $A_i = a_i a_i^T \in \S^r_+$ and $B_j = b_j b_j^T$ for $i=1,\dots,p$ and $j=1,\dots,q$. Then $A_i, B_j$ give a psd factorization of $M\hadprod M$ of size $r$. Indeed:
\[ \langle A_i, B_j \rangle = \langle a_i a_i^T, b_j b_j^T \rangle = (\langle a_i, b_j \rangle)^2 = M_{ij}^2 = (M\hadprod M)_{ij}. \]
Hence $\rankpsd (M\hadprod M) \leq \rank(M)$.
\end{itemize}
\end{proof}


The next theorem analyzes the psd rank of block-triangular matrices (the result below was also found independently by G{\'a}bor Braun and Sebastian Pokutta as well as in \cite{TroyLeeHermitianPsdRank}):

\begin{theorem}
\label{thm:blktri-ineq}
Let $P \in \RR^{p_1\times q_1}_+, Q \in \RR^{p_2\times q_1}_+, R\in \RR^{p_2\times q_2}_+$ be nonnegative matrices and let $M$ be the block matrix of size $(p_1+p_2)\times (q_1+q_2)$:
\[ M = \begin{bmatrix} P & 0\\ Q & R \end{bmatrix}. \]
Then
\begin{equation}
\label{eq:blktri-ineq}
 \rankpsd(M) \geq \rankpsd(P) + \rankpsd(R).
\end{equation}
Furthermore, when $Q = 0$ we have equality.
\end{theorem}
\begin{proof}
We first show the inequality \eqref{eq:blktri-ineq}. Assume the matrix $M$ has a psd factorization of size $k$ where the $p_1+p_2$ row factors are called $A_1,\dots,A_{p_1},\widehat{A_1},\dots,\widehat{A_{p_2}} \in \S^k_+$ and the $q_1+q_2$ column factors are $B_1,\dots,B_{q_1},\widehat{B_1},\dots,\widehat{B_{q_2}} \in \S^k_+$. Since the upper-right block of $M$ is zero, we have for all $(i,j) \in [p_1]\times [q_2]$, $\langle A_i, \widehat{B_j} \rangle = 0$. Hence, by Proposition \ref{prop:orthpsd}, $A_i \widehat{B_j} = 0$. Thus if we let $F=\sum_{i=1}^{q_2} \range(\widehat{B_j}) \subseteq \RR^k$, we have that $F \subseteq \ker(A_i)$ for all $i=1,\dots,p_1$. Since $A_i$ is symmetric this is equivalent to $\range(A_i) \subseteq F^{\perp}$.
Let $U$ be an orthonormal matrix whose columns consist of an orthonormal basis for $F^{\perp}$ concatenated with an orthonormal basis of $F$. Since $\range(A_i) \subseteq F^{\perp}$ we know that $A_i$ has the form:
\begin{equation}
\label{eq:Aiblk} A_i = U \begin{bmatrix} A'_i & 0\\ 0 & 0 \end{bmatrix} U^T \end{equation}
where $A'_i$ is of size $d$ where $d = \dim F^{\perp}$. Furthermore, since $\range(\widehat{B_j}) \subseteq F$ we have:
\begin{equation}  \widehat{B_j} = U \begin{bmatrix} 0 & 0\\ 0 & \widehat{B_j'} \end{bmatrix} U^T\end{equation}
where $\widehat{B_j'}$ is of size $k-d=\dim F$. Note that if we conjugate all the factors of the psd factorization of $M$ by $U$ (i.e., replace $A_i$ by $U^T A_i U$, etc.) we get another valid psd factorization of $M$ of the same size. Thus we can assume without loss of generality that $U=I$ and that $A_i$ and $\widehat{B_j}$ are block-diagonal.

If we now let $B_j'$ be the upper-left $d\times d$ block of $B_j$, then we have $P_{ij} = \langle A_i, B_j \rangle = \langle A_i', B_j' \rangle$, since $A_i$ has the block-diagonal structure \eqref{eq:Aiblk} (with $U=I$). Thus this shows that $\rankpsd(P) \leq d$.
Similarly, if we let $\widehat{A_i'}$ be the bottom-right $(k-d)\times (k-d)$ block of $\widehat{A_i}$, then we get $R_{ij} = \langle \widehat{A_i}, \widehat{B_j} \rangle = \langle \widehat{A_i'}, \widehat{B_j'} \rangle$ and thus $\rankpsd(R) \leq k-d$. Thus we finally get that $\rankpsd(P) + \rankpsd(R) \leq d+(k-d) = k$ which is the inequality we want.

We now show that when $Q=0$ we have $\rankpsd(M) = \rankpsd(P)+\rankpsd(R)$. Indeed let $P_{ij} = \langle C_i, D_j \rangle$ and $R_{ij} = \langle E_i, F_j \rangle$ be psd factorizations of $P$ and $R$ respectively of size $\rankpsd(P)$ and $\rankpsd(R)$.
Define:
\[ A_i = \begin{bmatrix} C_i & 0\\ 0 & 0 \end{bmatrix} \quad \forall i \in [p_1], \quad \widehat{A}_i = \begin{bmatrix} 0 & 0\\ 0 & E_i \end{bmatrix} \quad \forall i \in [p_2] \]
and
\[ B_j = \begin{bmatrix} D_j & 0\\ 0 & 0 \end{bmatrix} \quad \forall j \in [q_1],  \quad
\widehat{B}_j = \begin{bmatrix} 0 & 0\\ 0 & F_j \end{bmatrix} \quad \forall j \in [q_2]. \]
It is easy to see that the factors $A_1,\dots,A_{p_1},\widehat{A}_1,\dots,\widehat{A}_{p_2}$ and $B_1,\dots,B_{q_1},\widehat{B}_1,\dots,\widehat{B}_{q_2}$ give a psd factorization of the block-diagonal matrix $\begin{bmatrix} P & 0\\ 0 & R \end{bmatrix}$ of size $\rankpsd(P) + \rankpsd(R)$. Thus this shows, together with the inequality proved above, that
\[ \rankpsd \begin{bmatrix} P & 0\\ 0 & R \end{bmatrix} = \rankpsd(P) + \rankpsd(R). \]
\end{proof}

\begin{example}
A consequence of Theorem \ref{thm:blktri-ineq} is that the psd rank of the identity matrix $I_n$ is equal to $n$, since $I_n = \diag(1,\dots,1)$. In fact more generally the psd rank of a nonnegative diagonal matrix is equal to the number of nonzero diagonal elements.
\end{example}

\paragraph{Kronecker product} \label{sec:kron}
 The Kronecker (tensor) product of two matrices $M \in \RR^{p\times q}$ and $N \in \RR^{p'\times q'}$ is the $pp'\times qq'$ matrix $M\otimes N$ defined by:
\[ M\otimes N =
\begin{bmatrix}
M_{11} N & \dots & M_{1q} N\\
\vdots &       & \vdots\\
M_{p1} N & \dots & M_{pq} N
\end{bmatrix}.
\]
It is well-known that the rank of the Kronecker product $M\otimes N$ is equal to the product of the ranks of $M$ and $N$: $\rank(M\otimes N) = \rank(M) \rank(N)$.
 A natural question is to know whether the same is true for the psd rank.  In \cite{TroyLeeHermitianPsdRank} the authors give a counterexample to this, where they show that the psd rank of $M\otimes N$ can be strictly smaller than $\rankpsd(M)\rankpsd(N)$ (note that the inequality $\rankpsd (M\otimes N) \leq \rankpsd(M) \rankpsd(N)$ is always true).

\subsection{Lower semicontinuity of psd rank} In this subsection we show that for any $k\in \NN$, the set of matrices of psd rank $\leq k$ is closed (under the standard topology in $\RR^{p\times q}$). We prove the following:
\begin{theorem}
\label{thm:lsc}
Let $(M^n)_{n \in \NN}$ be a sequence of nonnegative matrices
converging to $M \in \RR^{p\times q}_+$ such that $\rankpsd (M^n) \leq
k$ for all $n \in \NN$. Then $\rankpsd(M) \leq k$.
\end{theorem}
\begin{proof}
The main ingredient to prove this result is to show that the factors $A_i, B_j$ in a psd factorization can always be chosen to be bounded. We have the following lemma:
\begin{lemma}
\label{lem:simplescaling}
Let $M \in \RR^{p\times q}_+$ and assume that $M$ has a psd factorization of size $k$. Then $M$ admits a psd factorization $M_{ij} =\langle A_i, B_j \rangle$ of size $k$ where the factors satisfy $\trace(A_i)\leq k$ and $\trace(B_j) = \sum_{i=1}^p M_{ij}$.
\end{lemma}
\begin{proof}
We defer the proof of this Lemma to Section \ref{sec:properties} where we discuss in more detail the issue of scaling the factors in a psd factorization.
\end{proof}
Let $(M^n)$ be a sequence of nonnegative matrices converging to $M$. Since $(M^n)$ is a convergent sequence the entries of $M^n$ are all bounded from above by some positive constant (independent of $n$).
The previous lemma shows that for each $n$, one can find a psd factorization of $M^n$ of the form:
\[ (M^n)_{ij} = \langle A^n_i, B^n_j \rangle \]
where $A^n_i, B^n_j \in \S^k_+$ and such that the sequences $(A^n_i)_{n \in \NN}$ and $(B^n_j)_{n \in \NN}$ are all bounded in $\S^k$. Thus one can extract convergent subsequences $(A^{\phi(n)}_i)$, $(B^{\phi(n)}_j)$ where $\phi:\NN\rightarrow \NN$ is increasing and $A^{\phi(n)}_i \rightarrow A_i$ and $B^{\phi(n)}_j \rightarrow B_j$ when $n\rightarrow +\infty$. Since $\S^k_+$ is closed we have $A_i, B_j \in \S^k_+$ and we get:
\[ M_{ij} = \lim_{n \rightarrow +\infty} \langle A^{\phi(n)}_i, B^{\phi(n)}_j \rangle = \langle A_i, B_j \rangle \]
which is a valid psd factorization of $M$ of size $k$. Thus $\rankpsd(M) \leq k$.
\end{proof}
\begin{remark}
The result above shows that the function $\rankpsd:\RR^{p\times q}_+ \rightarrow \NN$ is \emph{lower semi-continuous}, i.e., for any convergent sequence $M^n \rightarrow M$ it holds:
\[ \rankpsd (M) \leq \liminf_{n\rightarrow +\infty} \rankpsd (M^n). \]
It is well-known that the usual rank 
is also lower-semicontinuous, as well as the nonnegative rank (cf.  \cite{bocci2011perturbation} for the lower semi-continuity property of the nonnegative rank).
However, some notions of rank can fail to have this property. A well-known example is the rank of tensors of order $\geq 3$ which is not lower-semicontinuous, giving rise to the notion of border rank.
\end{remark}
%

\section{Motivation and examples}
\label{sec:motivation}
\subsection{Geometric interpretation} \label{subsec:geo interpretation}

In this section we discuss the geometric interpretation of the psd rank. This interpretation was in fact the original motivation that led to the definition of the psd rank in \cite{gouveia2011lifts}.

\emph{Semidefinite programming} is the problem of optimizing a linear function over an affine slice of the psd cone:
\[ \text{minimize} \; L(X) \; \text{ subject to } \; X \in  \S^k_+ \cap \cL \]
where $L:\S^k \rightarrow \RR$ is a linear function and $\cL$ is an affine subspace of $\S^k$. The feasible set $\S^k_+ \cap \cL$ of a semidefinite program is known as a \emph{spectrahedron} and can also be written as the solution set of a linear matrix inequality $\{x \in \RR^d : A_0 + x_1 A_1 + \dots + x_d A_d \succeq 0 \}$ where the $A_i$ are symmetric matrices that span the subspace $\cL$.
Semidefinite programs can be solved to arbitrary precision in polynomial-time, and have many applications in different areas of science and engineering \cite{vandenberghe1996semidefinite}.

Let $P \subset \RR^n$ be a polytope and assume we want to minimize a linear function $\ell$ over $P$, i.e., we want to compute $\min \{ \ell(x) : x \in P \}$. Observe that if $P$ admits a representation of the form
\begin{equation}
 \label{eq:psdlift}
 P = \pi(\S^k_+ \cap \cL)
\end{equation} where $\cL \subset \S^k$ is an affine subspace and $\pi$ is a linear map, then one can write the linear optimization problem over $P$ as a semidefinite program of size $k$, since:
\[ \min_{x \in P} \ell(x) = \min_{y \in \S^k_+ \cap \cL} \ell\circ \pi(y) \]
and $\ell \circ \pi$ is linear. A representation of the polytope $P$ of the form \eqref{eq:psdlift} is called a \emph{psd lift} of size $k$. Such a representation is interesting in practice when the size $d$ of the lift is much smaller than the number of facets of $P$, which is the size of the trivial representation of $P$ using linear inequalities.
A natural question to ask is thus: what is the smallest $k$ such that $P$ admits a psd lift of size $k$?

It turns out that the answer to this question is tightly related to the psd rank considered in this paper. For this we need to introduce the notion of a slack matrix:
\begin{definition}
\label{def:slack}
Let $P \subset \RR^n$ be a polytope (i.e., a bounded polyhedron) and $Q \subset \RR^n$ be a polyhedron with $P\subseteq Q$. Let $x_1,\dots,x_v$ be such that $P=\conv(x_1,\dots,x_v)$ and let $a_j \in \RR^n,b_j \in \RR$, $(j=1,\dots,f)$ be such that $Q=\{x \in \RR^n : a_j^T x \leq b_j \; \forall j\}$. Then the slack matrix of the pair $P,Q$, denoted $S_{P,Q}$ is the nonnegative $v\times f$ matrix whose $(i,j)$-th entry is $b_j - a_j^T x_i$. When $P=Q$ we write $S_{P} := S_{P,P}$ and we call it the slack matrix of $P$.
\end{definition}
\begin{remark}
Note that the entries of slack matrix of $S_{P,Q}$ can depend on the inequality description of $Q$ and the vertex description of $P$ (e.g., different scalings, redundant inequalities), however it is not hard to see that these do not affect the various ranks of the matrix, namely the usual rank, nonnegative rank and psd rank. Thus we will often refer to a slack matrix of a pair $P,Q$ as ``the'' slack matrix of $P,Q$.
\end{remark}
The next theorem gives an answer to the question of psd lifts formulated above, using the psd rank: it shows that the size of the smallest psd lift of a polytope $P$ is equal to  the psd rank of the slack matrix $S_P$ of $P$ (this is the case $P=Q$ of the statement below).
\begin{theorem}[cf. Proposition 3.6 in \cite{gouveia2013worst}]
\label{thm:psdlift}
Let $P\subset \RR^n$ be a polytope and $Q\subset \RR^n$ be a polyhedron such that $P\subseteq Q$, and let $S_{P,Q}$ be the slack matrix of the pair $P,Q$ (cf. Definition \ref{def:slack}). Then $\rankpsd S_{P,Q}$ is the smallest integer $k$ for which there exists an affine subspace $\cL$ of $\S^k$ and a linear map $\pi$ such that $P\subseteq \pi(\S^k_+ \cap \cL) \subseteq Q$.
\end{theorem}
\begin{proof}[Sketch of proof]
Let $k = \rankpsd S_{P,Q}$. We first show how to construct a spectrahedron $\S^k_+ \cap \cL$ of size $k$ such that $P \subseteq \pi(\S^k_+ \cap \cL) \subseteq Q$ for some linear map $\pi$. Let $x_1,\dots,x_v$ be the vertices of $P$ and let $Q = \{x \in \RR^n : a_j^T x \leq b_j \, \forall j=1,\dots,f\}$ be a facet description of $Q$. Let $A_1,\dots,A_v,B_1,\dots,B_f \in \S^k_+$ be a psd factorization of $S_{P,Q}$ of size $k$:
\[ b_j - a_j^T x_i = \langle A_i, B_j \rangle \quad \forall i=1,\dots,v, \; j=1,\dots,f. \]
Consider the convex set $C$:
\begin{equation}
\label{eq:defC}
C = \{x \in \RR^n : \exists A \in \S^k_+ \text{ such that } b_j - a_j^T x = \langle A , B_j \rangle \; \forall j=1,\dots,f \}.
\end{equation}
It is easy to verify that $C$ is contained between $P$ and $Q$: indeed $C \subseteq Q$ because any $x \in C$ satisfies $b_j - a_j^T x \geq 0$ for all $j=1,\dots,f$; also $P \subseteq C$ because the vertices $x_i$ of $P$ satisfy \eqref{eq:defC} with $A = A_i$.
Also it is not too difficult to show that $C$ can be expressed in the desired form $C = \pi(\S^k_+ \cap \cL)$ where $\cL$ is an affine subspace of $\S^k$ and $\pi$ is a linear projection map (we refer to \cite[Proposition 3.6]{gouveia2013worst} for the details). Thus this proves the first direction.

Assume now that we can write $P \subseteq \pi(\S^k_+ \cap \cL) \subseteq Q$ where $\cL$ is an affine subspace of $\S^k$ and $\pi$ is a linear map. We show how to construct a psd factorization of $S_{P,Q}$ of size $k$. Let $C = \pi(\S^k_+ \cap \cL)$. Using a suitable choice of basis for $\cL$, we can assume that $C$ has the form:
\[ C = \{ x \in \RR^n \; : \; \exists y \in \RR^m \text{ such that } T(x,y) \in \S^k_+ \} \]
where $T : \RR^n\times \RR^m \rightarrow \S^k$ is an affine linear map (i.e., $T$ has the form $T(x,y) = U_0 + x_1 U_1 + \dots + x_n U_n + y_1 V_1 + \dots + y_m V_m$ for some $U_0,U_1,\dots,U_n,V_1,\dots,V_m \in \S^k$). Observe that since $C \subseteq Q$ we have for any $(x,y) \in \RR^n \times \RR^m$:
\[ T(x,y) \in \S^k_+ \; \Rightarrow \; b_j - a_j^T x \geq 0 \; \forall j=1,\dots,f. \]
By Farkas' lemma this means that, for any $j=1,\dots,f$, there exists $B_j \in \S^k_+$ such that:
\[
b_j - a_j^T x = \langle T(x,y), B_j \rangle \quad \forall (x,y) \in \RR^{n}\times \RR^m.
\]
Furthermore, since $P \subseteq C$ we know that for any $x_i$ vertex of $P$ there exists $y_i$ such that $T(x_i,y_i) \in \S^k_+$. Thus if we let $A_i = T(x_i,y_i)$ then we get the following psd factorization of size $k$ of the slack matrix $S_{P,Q}$:
\[ b_j - a_j^T x_i = \langle A_i, B_j \rangle \quad \forall i=1,\dots,v,j=1,\dots,f. \]
This completes the proof.
%
%
%
%
\end{proof}
Note that the proof of Theorem \ref{thm:psdlift} is constructive: it shows how to construct the spectrahedron $\S^k_+ \cap \cL$ and the linear map $\pi$ from a psd factorization of $S_{P,Q}$, and vice-versa.

\begin{remark}
\label{rem:slackmatrices}
 The geometric interpretation of the psd rank given in Theorem \ref{thm:psdlift} can be used to study the psd rank of any \emph{arbitrary} nonnegative matrix $M$, since one can show that any nonnegative matrix $M$ is the slack matrix of some pair of polytopes $P,Q$.  We use this geometric interpretation in Section \ref{sec:psdrank2_convexprogramming} to show that one can use semidefinite programming to decide if $\rankpsd M \leq 2$. Note that if $P,Q$ are full-dimensional polytopes in $\RR^n$, then the usual rank of the slack matrix $S_{P,Q}$ is equal to $n+1$. For example if $M$ is a nonnegative matrix with rank two, then it is the slack matrix of two nested intervals in the real line. It thus follows easily from this geometric interpretation and from Theorem \ref{thm:psdlift}  that the psd rank of any rank-two matrix is equal to 2 (this was already shown in Proposition \ref{prop:rank1rank2} using a result from \cite{cohen1993nonnegative}).
\end{remark}

We now illustrate Theorem \ref{thm:psdlift} using two simple examples.

\begin{example} Let $P = Q = [-1,1]^2$ be the square in the plane.
 The polytope $P$ has 4 facets and 4 vertices and the slack matrix of $P$ can be shown to be equal to the following $4\times 4$ matrix:
\begin{equation}
\label{eq:Msquare}
 M = \begin{bmatrix}
1 & 1 & 0 & 0\\
0 & 1 & 1 & 0\\
0 & 0 & 1 & 1\\
1 & 0 & 0 & 1
\end{bmatrix}.
 \end{equation}
One can construct the following psd factorization of $M$ of size 3, showing that $\rankpsd M \leq 3$: $M_{ij} = \langle A_i, B_j \rangle$ where $A_i = u_i u_i^T \in \S^3_+, B_j = v_j v_j^T \in \S^3_+$, $i,j=1,\dots,4$ with:
\[
\begin{array}{c}
u_1 = (1,0,0), \;\; u_2 = (0,1,0), \;\; u_3 = (0,0,1), \;\; u_4 = (1,1,1)\\
v_1 = (1,0,0), \;\; v_2 = (1,-1,0), \;\; v_3 = (0,1,-1), \;\; v_4 = (0,0,1).
\end{array}
 \]
 Thus by Theorem \ref{thm:psdlift}, this means that one can represent the polytope $P=[-1,1]^2$ as the linear image of a spectrahedron of size $3$. One can in fact show that $P$ is the projection onto the $(x,y)$ coordinates of the following spectrahedron $T$ of size 3:
\[ T = \left\{ (x,y,z) \in \RR^3 \; : \; \begin{bmatrix} 1 & x & y\\ x & 1 & z\\ y & z & 1 \end{bmatrix} \succeq 0 \right\}. \]
The spectrahedron $T$ (also known as the \emph{elliptope}) is depicted in Figure \ref{fig:square_lift_elliptope}.
Note that no smaller representation of the square $[-1,1]^2$ is possible: it was proved in \cite{gouveia2013polytopes} that the psd rank of any $n$-dimensional polytope is at least $n+1$ which in this case means that $\rankpsd M \geq 2+1=3$.

\begin{figure}[ht]
  \includegraphics[width=3.5cm]{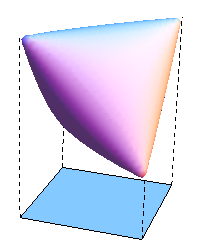}
  \caption{A psd lift of the square $[-1,1]^2$ of size 3: the elliptope $\{ X \in \S^3_+ \; : \; \diag(X) = \onevec \}$ linearly projects onto the square $[-1,1]^2$.}
  \label{fig:square_lift_elliptope}
\end{figure}

\end{example}

\begin{example} We now give another illustration of Theorem \ref{thm:psdlift} where the polytopes $P$ and $Q$ are different. Let $Q = [-1,1]^2$ and let now $P = [-a,a]\times [-b,b]$ be the rectangle centered at the origin with side lengths $2a$ and $2b$ with $0 \leq a,b \leq 1$ (cf. Figure \ref{fig:nested_rectangles_problem}).
The slack matrix of the pair $P,Q$ can be easily computed and is given by:
\begin{equation}
\label{eq:nested_rectangles_matrix}
M =
\begin{bmatrix}
1+a & 1+b & 1-a & 1-b\\
1-a & 1+b & 1+a & 1-b\\
1-a & 1-b & 1+a & 1+b\\
1+a & 1-b & 1-a & 1+b
\end{bmatrix}.
\end{equation}

\begin{figure}[ht]
\centering
  \includegraphics[width=7cm]{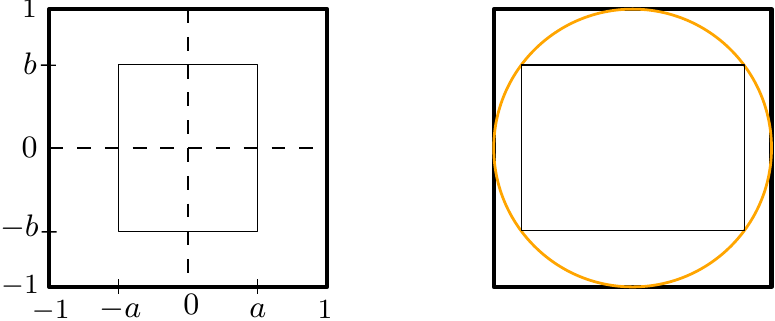}
  \caption{The geometric problem associated to the slack matrix $M$ of Equation \eqref{eq:nested_rectangles_matrix}. The inner polytope is $P=[-a,a]\times [-b,b]$ and the outer polytope is $Q=[-1,1]^2$. The right figure shows an instance where there exists an ellipse $E$ such that $P \subset E \subset Q$.}
  \label{fig:nested_rectangles_problem}
\end{figure}

Theorem \ref{thm:psdlift} says that $\rankpsd M$ is equal to the smallest size of a spectrahedron which has a linear projection that is contained between $P$ and $Q$. In the previous example we saw a spectrahedron of size 3 which projects onto $Q=[-1,1]^2$ and thus this shows that $\rankpsd M \leq 3$ for all $a,b \leq 1$. It is natural to ask whether the psd rank of $M$ can be equal to 2 for some values of $a,b$? One can actually show that the psd rank of $M$ is equal to 2 if, and only if, there is an ellipse $E$ such that $P \subseteq E \subseteq Q$ (cf. e.g., \cite[Section 4]{gouveia2013worst}). It is not hard to see that such an ellipse exists if and only if $a^2 + b^2 \leq 1$. Thus we have the following:
\[ \rankpsd M = \begin{cases} 3 & \text{if } a^2 + b^2 > 1\\ 2 & \text{if } 0< a^2 + b^2 \leq 1\\ 1 & \text{if } a=b=0 \end{cases}. \]

%
\end{example}

%
%

\subsection{Information theoretic interpretation}
We now describe a different application of the psd rank in the area of quantum information theory.
Let $M$ be a $p\times q$ nonnegative matrix and assume that the entries of $M$ all sum up to 1. Then $M$ can be seen as a joint probability distribution of a pair of random variables $(X,Y)$, where $M_{ij} = P(X=i,Y=j)$. It is known that a nonnegative factorization of $M$ can be interpreted as a representation of $(X,Y)$ as a mixture of independent random variables, see e.g., \cite[Section 6]{cohen1993nonnegative}. As such the \emph{nonnegative rank} of $M$ defines a certain measure of correlation between random variables $X$ and $Y$. In this section we show that a similar interpretation of the psd rank holds, and that $\rankpsd M$ gives a measure of correlation between $X$ and $Y$ in terms of \emph{quantum information theory}.  This quantum interpretation of the psd rank was first pointed out in the paper \cite{jain2013efficient}.

\begin{remark}
We remark that this is not the only known interpretation of the psd rank in quantum information theory: in \cite{fiorini2012lowerbound} the authors show that the psd rank of a matrix $M$ is equal to \emph{the one-way quantum communication complexity of computing the matrix $M$ in expectation}. Also the psd rank is tightly related to the problem of determining the smallest dimension of a Hilbert space that explains certain measured correlations, see e.g., \cite{brunner2008testing,wehner2008lower} for more details. For space reasons, however, we focus only on the interpretation of \cite{jain2013efficient} in terms of correlation of two random variables $X,Y$.
\end{remark}


\subsubsection{Correlation generation} Given a pair of random variables $(X,Y)$, consider the following \emph{correlation generation} game: Two parties, Alice and Bob (for short, $A$ and $B$), want to generate samples from the pair of variables $(X,Y)$. Alice outputs samples from $X$ and Bob outputs samples from $Y$ and they want to do it in such a way that the samples follow the \emph{joint} distribution of $(X,Y)$. Note that if $X$ and $Y$ were independent each party could independently sample from the marginals and they would successfully achieve their objective. However if $X$ and $Y$ are correlated then the two parties $A$ and $B$ must either communicate together or share some common information in order to achieve their task. We will show here that the minimum amount of \emph{quantum} information that $A$ and $B$ need to have in common is precisely $\log \rankpsd M$ where $M$ is the matrix giving the joint distribution of $(X,Y)$. Thus this shows that $\log \rankpsd M$ gives a measure of the correlation between $X$ and $Y$ in terms of quantum information theory.

\begin{figure}[ht]
  \centering
  \includegraphics[width=5cm]{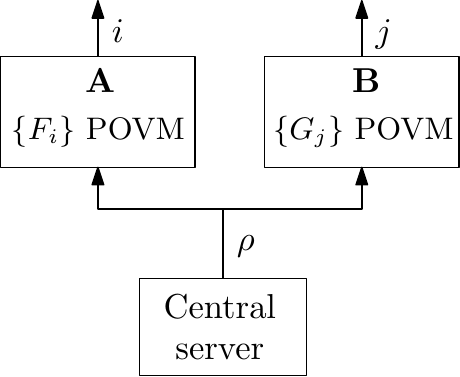}
  \caption{Quantum correlation generation problem: $A$ and $B$ generate samples from the joint distribution $(X,Y)$ using shared information provided by a central server. The psd rank characterizes the minimum amount of quantum information that has to be shared between the two parties.}
  \label{fig:correlation_generation_protocol}
\end{figure}

We first recall some basic terminology from quantum information theory. The state of a finite-dimensional quantum system is represented by a Hermitian positive semidefinite matrix $\rho$ with trace 1, called a \emph{density operator}. For convenience, we will work here with real symmetric matrices (instead of complex Hermitian) since our definition of psd rank involves real symmetric matrices. The state of a \emph{bipartite} system is described by a density operator $\rho$ of dimension $n_1 n_2$ where $n_1$ is the dimension of the first part (or subsystem) and $n_2$ is the dimension of the second part.
Measurements in quantum mechanics are formalized using the concept of \emph{POVM} (short for Positive Operator-Valued Measure). A POVM is a finite collection of psd matrices $F_1,\dots,F_p \in \S^k_+$ that satisfy $\sum_{i=1}^p F_i = I_k$ where $I_k$ is the identity matrix. The outcome of measuring a state $\rho$ using a POVM $\{F_1,\dots,F_p\}$ is $i \in \{1,\dots,p\}$ with probability $\trace(F_i \rho)$. Note that $\trace(F_i \rho) \geq 0$ and $\sum_{i=1}^p \trace(F_i \rho) = \trace(\rho) = 1$.

Let $M$ be a $p\times q$ nonnegative matrix such that $M_{ij} = P(X=i,Y=j)$ where $(X,Y)$ is a pair of random variables. Assume we have a decomposition of $M$ of the form:
\begin{equation}
 \label{eq:quantum-decomp}
 M_{ij} = \trace((F_i \otimes G_j) \rho) \quad \forall i=1,\dots,p, \; j=1,\dots,q
\end{equation}
where $\rho \in \S^{k^2}$ is a bipartite quantum state (where each subsystem has dimension $k$) and where $\{F_i\}$ and $\{G_j\}$ are POVMs, i.e., $F_i,G_j \in \S^k_+$ and $\sum_{i=1}^p F_i = \sum_{j=1}^q G_j = I_k$. The notation $\otimes$ here indicates Kronecker product. If there is such a decomposition of $M$, then one can produce samples from the pair $(X,Y)$ using the help of a central server as follows (cf. Figure \ref{fig:correlation_generation_protocol}): the central server sends the first part of the state $\rho$ to Alice and the second part to Bob (each part has dimension $k$). Alice and Bob perform measurements using POVMs $\{F_i\}$ and $\{G_j\}$ respectively and output the outcomes $i$ and $j$ of their measurements. The laws of quantum mechanics say that the outcome $(i,j)$ occurs with probability $\trace((F_i\otimes G_j) \rho)$. Identity \eqref{eq:quantum-decomp} thus guarantees that the outputs of Alice and Bob follow the joint distribution of $(X,Y)$.

The cost of the protocol described above is the number of quantum bits communicated by the central server to Alice and Bob, which in this case is $\log k$ (a quantum system of dimension $n$ is represented using $\log n$ qubits). We are thus interested in the smallest $k$ for which a decomposition of $M$ of the form \eqref{eq:quantum-decomp} exist. It turns out that this is equal to the psd rank of $M$ as we show in the next proposition.

\begin{proposition} \cite{jain2013efficient}
Let $M$ be a $p \times q$ nonnegative matrix where all the entries sum up to one. Let $k \geq 1$. Then the following are equivalent:\\
(i) $M$ admits a psd factorization of size $k$, i.e., there exist $A_i, B_j \in \S^k_+$ such that $M_{ij} = \langle A_i, B_j \rangle$ for all $i=1,\dots,p$ and $j=1,\dots,q$.\\
(ii) There is a quantum protocol for the correlation generation problem using $\log k$ qubits, i.e., $M$ admits a factorization of the form \eqref{eq:quantum-decomp} of size $k$.
\end{proposition}
\begin{proof}
(ii) $\Rightarrow$ (i): Assume we have a decomposition of $M$ of the form
\[ M_{ij} = \trace((F_i\otimes G_j)\rho) \quad \forall i=1,\dots,p, \; j=1,\dots,q \]
where $F_i \in \S^k_+$ and $G_j \in \S^k_+$ are psd matrices such that $\sum_{i=1}^p F_i = \sum_{j=1}^q G_j = I_k$ and $\rho \in \S^{k^2}$ is psd such that $\trace(\rho) = 1$. Assume for simplicity that $\rho$ is rank-one, i.e., $\rho = \psi \psi^T$ where $\psi \in \RR^{k^2}$ (the general case is very similar). Since $\psi \in \RR^{k^2} \cong \RR^{k\times k}$ we know that $\rank \psi \leq k$ and so we can write:
\[ \psi = \sum_{s=1}^k v_s \otimes w_s, \]
where $v_k, w_k \in \RR^k$. Let $V$ and $W$ be the matrices with the $v_s$ and $w_s$ in columns, i.e., $V = [v_1|\dots|v_k]$ and $W = [w_1|\dots|w_k]$. Define $A_i = V^T F_i V$ and $B_j = W^T G_j W$ for $i=1,\dots,p$ and $j=1,\dots,q$. Clearly $A_i$ and $B_j$ are psd and have size $k$. We claim that $A_i,B_j$ give a psd factorization of $M$ of size $k$.
Indeed we have:
\[
\begin{aligned}
\trace((F_i \otimes G_j)\psi \psi^T) &= \psi^T (F_i \otimes G_j) \psi\\
&= \sum_{1\leq s,t \leq k} (v_s \otimes w_s)^T (F_i \otimes G_j) (v_t \otimes w_t)\\
&\overset{(*)}{=} \sum_{1\leq s,t \leq k} (v_s^T F_i v_t) (w_s^T G_j w_t)\\
&= \langle V^T F_i V, W^T G_j W \rangle = \langle A_i, B_j \rangle.
\end{aligned}
\]
where in $(*)$ we used the mixed-product property of the Kronecker product $(A\otimes B) (C\otimes D) = (AC)\otimes (BD)$.

(i) $\Rightarrow$ (ii): We now prove the other direction. Assume we have a psd factorization of $M$ of the form $M_{ij} = \langle A_i, B_j \rangle$ where $A_i,B_j \in \S^k_+$. We show how to construct a factorization of the form \eqref{eq:quantum-decomp}. Let $\Sigma_A,\Sigma_B$ be defined respectively by $\Sigma_A = \sum_{i=1}^p A_i$ and $\Sigma_B = \sum_{j=1}^q B_j$. Note that $\Sigma_A$ and $\Sigma_B$ can be assumed to be invertible (otherwise we can reduce the size of the psd factorization). Consider the matrices $F_i$ and $G_j$ defined by:
\[ F_i = \Sigma_A^{-1/2} A_i \Sigma_A^{-1/2} \;\; (i=1,\dots,p) \quad \text{ and } \quad G_j = \Sigma_B^{-1/2} B_j \Sigma_B^{-1/2} \;\; (j=1,\dots,q). \]
Then $F_i,G_j \succeq 0$ and $\sum_{i=1}^p F_i = \sum_{j=1}^q G_j = I_k$.
We now construct the state $\rho \in \S^{k^2}$ of the protocol. To do so recall that we have the following simple fact: If $A$ and $B$ are symmetric matrices of size $k$, then \[ \trace(AB) = e^T (A\otimes B) e \]
where $e = \text{vec}(I_k) \in \RR^{k^2}$ is the vector obtained by stacking all the columns of $I_k$ into a single column of dimension $k^2$. Let $\rho \in \S^{k^2}$ be defined by $\rho = \psi \psi^T$
where
\[ \psi = (\Sigma_{A}^{1/2} \otimes \Sigma_{B}^{1/2}) e. \]
First note that $\rho$ is a valid state and $\trace(\rho) = 1$ since
\[ \trace(\rho) = \psi^T \psi = e^T (\Sigma_A \otimes \Sigma_B) e = \trace(\Sigma_A \Sigma_B) = \sum_{\substack{1\leq i\leq p\\ 1\leq j\leq q}} \trace(A_i B_j) = \sum_{\substack{1\leq i\leq p\\ 1\leq j\leq q}} M_{ij} = 1. \]
We now claim that the choice of $\{F_i\},\{G_j\}$ and $\rho$ gives a valid decomposition of $M$ as in \eqref{eq:quantum-decomp}. Indeed we have:
\[
\begin{aligned}
\trace((F_i \otimes G_j) \rho) &= \psi^T(F_i \otimes G_j)\psi\\
&= e^T (\Sigma_{A}^{1/2} \otimes \Sigma_{B}^{1/2}) (\Sigma_A^{-1/2} A_i \Sigma_A^{-1/2} \otimes \Sigma_B^{-1/2} B_j \Sigma_B^{-1/2}) (\Sigma_{A}^{1/2} \otimes \Sigma_{B}^{1/2}) e\\
&= e^T (A_i \otimes B_j) e = \trace(A_i B_j) = M_{ij}.
\end{aligned}
 \]
\end{proof}


\section{Psd rank two and convex programming}
\label{sec:psdrank2_convexprogramming}
In Proposition \ref{prop:rank1rank2} we showed that if $M$ is a nonnegative matrix with $\rank(M) \leq 2$ then $\rankpsd(M) = \rank(M)$. When $\rank(M) = 3$, then inequalities \ref{prop:ineqsranks} imply that $\rankpsd(M) \geq 2$. In this section we show that one can decide whether $\rankpsd(M) = 2$ using semidefinite programming.

We saw in section \ref{sec:motivation} that any nonnegative matrix $M$ can always be interpreted as the slack matrix of a pair of polyhedra $P,Q$ where $P\subset Q$ and where $P$ is bounded. In fact one can always choose the outer polyhedron $Q$ to be bounded as well, as is explained for example in \cite[Theorem 1]{gillis2012geometric}:
\begin{lemma}
Let $M \in \RR^{p\times q}_+$ be a nonnegative matrix and assume that $M\onevec = \onevec$. Let $r=\rank M$. Then there exist polytopes $P,Q$ in $\RR^{r-1}$ (where $P$ and $Q$ are bounded) such that $P\subset Q$ and such that $M$ is the slack matrix of the pair $P,Q$.
\end{lemma}
\begin{proof}
The proof is in \cite{gillis2012geometric} and we reproduce it here for completeness.
In \cite{gillis2012geometric} it is shown that one can always find a factorization of $M$ of the form $M=AB$ where $A \in \RR^{p\times r}, B \in \RR^{r\times q}$ and $A\onevec=\onevec$ and $B\onevec = \onevec$. Write $A$ and $B$ as:
\[ A = \begin{bmatrix} a_1^T & t_1 \\ \vdots \\ a_p^T & t_p \end{bmatrix}
\quad
B = \begin{bmatrix} b_1^T \\ \vdots \\ b_{r}^T \end{bmatrix}, \]
where $a_i \in \RR^{r-1}$, $t_i\in \RR$ for $i=1,\dots,p$ and $b_j \in \RR^q$ for all $j=1,\dots,r$. Note that since $A\onevec = \onevec$ we have $t_i = 1 - \onevec^T a_i$.
Define the polytopes $P$ and $Q$ as follows:
\[ P = \conv(a_1,\dots,a_p) \subset \RR^{r-1} \]
and
\[ Q = \left\{ x \in \RR^{r-1} \; : \; \sum_{i=1}^{r-1} x_i b_i + \left(1-\sum_{i=1}^{r-1} x_i\right) b_r \geq 0 \right\}. \]
Note that $Q$ is defined using $q$ linear inequalities. It is not difficult to verify that $M$ is the slack matrix of the pair $P,Q$. It remains to show that $Q$ is bounded. Assume for contradiction that $x_0 + \alpha z \in Q$ for all $\alpha \geq 0$ where $x_0 \in Q$. Then one can show that this implies that
 \[ w := \sum_{i=1}^{r-1} z_i b_i - \left(\sum_{i=1}^{r-1} z_i\right) b_r \geq 0. \]
Note that we have $\onevec^T w = 0$ since $\onevec^T b_i = 1$ for all $i=1,\dots,r$. Thus since $w \geq 0$ and $\onevec ^T w = 0$, this means that $w=0$, i.e., $\sum_{i=1}^{r-1} z_i b_i - \left(\sum_{i=1}^{r-1} z_i\right) b_r=0$. Since $B$ is full-rank this necessarily means that $z=0$. We have thus shown that $Q$ is bounded.
\end{proof}

Assume that $\rank M = 3$ and let $P\subset Q \subset \RR^2$ be two polytopes in the plane such that $M$ is the slack matrix of $P$ with respect to $Q$. From \cite[Proposition 4.1]{gouveia2013worst}, we know that $\rankpsd M = 2$ if, and only if, there exists an ellipse $E$ such that $P \subset E \subset Q$. Since we have a vertex description of $P$, and a facet description of $Q$, this can be decided using semidefinite programming: Indeed, let $x_1,\dots,x_v$ be the vertices of $P$, and let $Q=\{x \in \RR^2 \; : \; Gx \leq h\}$ be a facet description of $Q$ where $G$ has $f$ rows.
%
One can show that there exists an ellipse sandwiched between $P$ and $Q$ if, and only if, there exist $A \in \S^2, b \in \RR^2$ and $c \in \RR$ such that:
\[
\begin{aligned}
1. \quad & A \succeq 0, \; \trace(A) = 1;\\
2. \quad& \begin{bmatrix} x_j \\ 1 \end{bmatrix}^T \begin{bmatrix} A & b\\ b^T & c \end{bmatrix} \begin{bmatrix} x_j\\ 1 \end{bmatrix} \leq 0 \quad \forall j=1,\dots,v;\\
3. \quad& \exists \lambda_i \geq 0 \; : \;
\begin{bmatrix} A & b\\ b^T & c \end{bmatrix}
 \succeq \lambda_i \begin{bmatrix} 0 & g_i^T/2\\ g_i/2 & -h_i \end{bmatrix} \quad \forall i=1,\dots,f.
\end{aligned}
\]
The ellipse $E$ that satisfies $P\subset E \subset Q$ is then defined by:
\[ E = \left\{ x \in \RR^2 \; : \; \begin{bmatrix} x \\ 1 \end{bmatrix}^T \begin{bmatrix} A & b\\ b^T & c \end{bmatrix} \begin{bmatrix} x\\ 1 \end{bmatrix} \leq 0\right\}. \]
Note that the constraint (2) above corresponds to the condition $P\subseteq E$ and the constraint (3) corresponds to $E \subseteq Q$. The latter uses the following result commonly known as the \emph{S-lemma} \cite[Appendix B]{boyd2004convex}:
\begin{lemma}\label{lem:s-lemma}
Let $A_i \in \S^n,b_i \in \RR^n,c_i \in \RR$ for $i=1,2$ and assume that the following implication holds for all $x \in \RR^n$:
\[ x^T A_1 x + 2b_1^T x + c_1 \leq 0 \; \Longrightarrow \; x^T A_2 x + 2b_2^T x + c_2 \leq 0. \]
Then there exists a $\lambda \geq 0$ such that:
\[
\begin{bmatrix} A_2 & b_2\\ b_2^T & c_2 \end{bmatrix}
\preceq
\lambda
\begin{bmatrix} A_1 & b_1\\ b_1^T & c_1 \end{bmatrix}.
\]
\end{lemma}

\section{Relationships between ranks}
\label{sec:relationships_between_ranks}
Recall from Proposition~\ref{prop:ineqsranks} that for a nonnegative matrix $M \in \RR^{p \times q}_+$,
\begin{align}
 \frac{1}{2}\sqrt{1+8 \,\rank(M)}-\frac{1}{2} \leq \rankpsd (M) \leq \rankplus (M). \label{eq:rank inequalities}
\end{align}
The first inequality is equivalent to saying that for all nonnegative matrices $M$,
\begin{align}
\rank (M) \leq {\rankpsd (M) + 1 \choose 2 }. \label{eq:easy rank psd rank}
\end{align}
This says that while rank may be
larger than psd rank, it cannot be much larger, since it is bounded
above by a quadratic function of the psd rank.
In this section, we examine the relationships between the three ranks present in
inequality~\eqref{eq:rank inequalities} and a fourth type of rank called
{\em square root rank}.  We begin by showing that all inequalities
in \eqref{eq:rank inequalities} can be tight. An easy example for the second inequality is the
$n \times n$ identity matrix for which $\rankpsd (I_n) = \rankplus (I_n) =n$.

\begin{example} [{\bf Derangement matrices}] \label{ex:derangement}

The $n \times n$ {\em derangement matrix} $D_n$ is the matrix with zeros on the diagonal and ones elsewhere. It verifies $\rank (D_n) = n$ for all $n$. Fix a positive integer $k$ and let $n := {k+1 \choose 2}$. We will exhibit a factorization of $D_n$ through $\PSD^k$ which will show that $\rankpsd (D_n) \leq k$, making the first inequality tight.


To construct a psd factorization of $D_n$ through $\PSD^k$ choose factors as follows:
For $i = 1,\ldots,k$, let $A_i = e_i e_i^T$ where $e_i$ is the $i$th standard basis vector in $\RR^k$.  Since ${k+1 \choose 2} = {k \choose 2} +k$, we need to define $\binom{k}{2}$ further
$A_i$ matrices.  Let
$F= \left[ \begin{array}{cc}
1 & -1 \\
-1 & 1 \end{array} \right] $.  For each $i,j \in \left\{1,\ldots,k\right\}$ with $i < j$, define $F_{i,j}$ to be equal to the $k \times k$ matrix that has its $i,j$ principal submatrix equal to $F$ and all other entries equal to $0$.  Let $A_{k+1} = F_{1,2}$, $A_{k+2} = F_{1,3}$, \ldots, $A_{2k} = F_{2,3}$,$A_{2k+1} = F_{2,4}$, and so on.
Now we define matrices $B_1,\ldots,B_n$ for the columns.  First, let $E$ be the $(k-1) \times (k-1)$ matrix with ones on the diagonal and $\frac{1}{2}$ everywhere else.  For $i = 1,\ldots,k$, let $B_i$ be the matrix whose $i$th row and column are identically zero and whose remaining entries form the matrix $E$.  For $i > k$, we obtain the matrix $B_i$ from $A_i$ by the following:  First change all nonzero entries and all diagonal entries of $A_i$ to ones.  Then change all remaining zero entries to $\frac{1}{2}$.  Call the resulting matrix $B_i$.
The matrices $A_i,B_j$ form a psd factorization of $D_n$.

We present the case $k=3$ below:
\[ D_6 = \left[ \begin{array}{cccccc}
0 & 1 & 1 & 1 & 1 & 1 \\
1 & 0 & 1 & 1 & 1 & 1 \\
1 & 1 & 0 & 1 & 1 & 1 \\
1 & 1 & 1 & 0 & 1 & 1 \\
1 & 1 & 1 & 1 & 0 & 1 \\
1 & 1 & 1 & 1 & 1 & 0 \end{array} \right] \]

Then $A_1,\ldots,A_6$ are:
{\tiny {
\[ \left[ \begin{array}{ccc}
1 & 0 & 0 \\
0 & 0 & 0 \\
0 & 0 & 0  \end{array} \right],
\left[ \begin{array}{ccc}
0 & 0 & 0 \\
0 & 1 & 0 \\
0 & 0 & 0  \end{array} \right],
\left[ \begin{array}{ccc}
0 & 0 & 0 \\
0 & 0 & 0 \\
0 & 0 & 1  \end{array} \right]
\left[ \begin{array}{ccc}
1 & -1 & 0 \\
-1 & 1 & 0 \\
0 & 0 & 0  \end{array} \right],
\left[ \begin{array}{ccc}
1 & 0 & -1 \\
0 & 0 & 0 \\
-1 & 0 & 1  \end{array} \right],
\left[ \begin{array}{ccc}
0 & 0 & 0 \\
0 & 1 & -1 \\
0 & -1 & 1  \end{array} \right], \]
}}

and $B_1,\ldots,B_6$ are:
{\tiny{
\[ \left[ \begin{array}{ccc}
0 & 0 & 0 \\
0 & 1 & \frac{1}{2} \\
0 & \frac{1}{2} & 1  \end{array} \right],
\left[ \begin{array}{ccc}
1 & 0 & \frac{1}{2} \\
0 & 0 & 0 \\
\frac{1}{2} & 0 & 1  \end{array} \right],
\left[ \begin{array}{ccc}
1 & \frac{1}{2} & 0 \\
\frac{1}{2} & 1 & 0 \\
0 & 0 & 0  \end{array} \right],
 \left[ \begin{array}{ccc}
1 & 1 & \frac{1}{2} \\
1 & 1 & \frac{1}{2} \\
\frac{1}{2} & \frac{1}{2} & 1  \end{array} \right],
\left[ \begin{array}{ccc}
1 & \frac{1}{2} & 1 \\
\frac{1}{2} & 1 & \frac{1}{2} \\
1 & \frac{1}{2} & 1  \end{array} \right],
\left[ \begin{array}{ccc}
1 & \frac{1}{2} & \frac{1}{2} \\
\frac{1}{2} & 1 & 1 \\
\frac{1}{2} & 1 & 1  \end{array} \right]. \]
}}

We showed that $\rankpsd (D_n) = k$ whenever $n = \binom{k+1}{2}$.  Now suppose that $n$ is strictly between $\binom{k}{2}$ and $\binom{k+1}{2}$.  Then the rank lower bound (first inequality in ~\eqref{eq:rank inequalities}) implies that $\rankpsd (D_n) > k-1$.  Since $D_n$ is a submatrix of $D_{\binom{k+1}{2}}$ for which we know a size $k$ psd factorization, $\rankpsd (D_n) \leq k$.  Thus, $\rankpsd (D_n) = k$ for these intermediate values of $n$, or equivalently, $$\rankpsd (D_n) = \textup{min} \left\{ k \,:\, n \leq {k+1 \choose 2} \right\} \textup{ for all } n.$$
\end{example}


\subsection{Square root rank: an upper bound for psd rank}
Given a nonnegative matrix $M$, let $\sqrt M$ denote a Hadamard square root of $M$ obtained by replacing each entry in $M$ by one of its two possible square roots.

\begin{definition} \label{def:sqrt rank} The {\bf square root rank} of a nonnegative matrix $M$, denoted as $\sqrtrank(M)$, is the minimum rank of a Hadamard square root of $M$.
\end{definition}

For a quick example of square root rank, note that the following matrix $M$ (of rank 3) has $\sqrtrank(M) = 2$ as evidenced by the shown square root.
\[ M = \left( \begin{array}{ccc} 1 & 0 & 1 \\ 0 & 1 & 4 \\ 1 & 1 & 1 \end{array} \right)
\longrightarrow \sqrt{M} = \left( \begin{array}{ccc} 1 & 0 & 1 \\ 0 & 1 & -2 \\ 1 & 1 & -1 \end{array} \right) \]

Recall from the proof of Theorem~\ref{thm:firstprops} (v)  that if a Hadamard square root of $M$ has rank $r$ then there is a psd factorization of $M$ by matrices of rank one lying in the psd cone $\PSD^r$. This implies the following corollary.

\begin{corollary}
\label{cor:upper bound from sqrt rank}
For a nonnegative matrix $M \in \RR^{p \times q}$, $\rankpsd (M) \leq \sqrtrank (M)$.
In particular, if $M$ is a $0/1$ matrix, then $\rankpsd (M) \leq \rank (M)$.
\end{corollary}


The second statement of Corollary~\ref{cor:upper bound from sqrt rank} says that if a matrix has only the two distinct entries $0$ and $1$, then its psd rank is bounded above by rank.
This was extended by Barvinok \cite{barvinok2012approximations} to an upper bound on the psd rank of a matrix in terms of its rank and number of distinct entries.

\begin{lemma}
\cite[Lemma~4.4]{barvinok2012approximations}
\label{lem:Barvinok}
\item Let $A = (a_{ij})$ be a real matrix and $f \,:\, \RR \rightarrow \RR$ be a polynomial of degree $k$. If $B  = (b_{ij})$ is such that $b_{ij} = f(a_{ij})$ for all $i,j$, then $$\rank (B) \leq \binom{k+\rank (A)}{k}.$$
\end{lemma}


\begin{corollary} \cite[Lemma 4.6]{barvinok2012approximations}\label{cor:Barvinok}
If the number of distinct entries in a nonnegative matrix $M$ does not exceed $k$, then $\rankpsd (M) \leq { {k-1 + \rank (M) } \choose {k-1}}$.
\end{corollary}

\begin{proof}
Let $\mathcal{M}$ be the set of distinct entries in $M$ and $\phi\,:\, \mathcal{M} \rightarrow \RR$ be the square root function.
Since $| \mathcal{M} | \leq k$, there exists a polynomial $f(t)$ of degree $k-1$ such that $\phi(t) = f(t)$ on $\mathcal{M}$.
Then by Lemma~\ref{lem:Barvinok} and Corollary~\ref{cor:upper bound from sqrt rank} we have  that
$$\rankpsd (M) \leq \rank (\sqrt{M}) \leq \binom{k-1+\rank (M)}{k-1}.$$
\end{proof}

While we strongly suspect that it is NP-hard to compute psd rank,
there is no proof of this fact at the moment. The situation is clearer
for square root rank.

\begin{theorem}
The square root rank of a nonnegative matrix is NP-hard to compute.
\end{theorem}

\begin{proof}
Recall that given a list of $n$ positive integers $a_1,\ldots,a_n$, the partition problem asks whether there exist sign choices $s_1,\ldots,s_n \in \{-1,1\}$ such that $\sum_{i=1}^{n}{s_i a_i} = 0$.  This problem is known to be NP-complete \cite{GareyJohnson}.

Given the integers $a_1,\ldots,a_n$, define an $(n+1) \times (n+1)$ matrix $A$ of the form:
\[ \left( \begin{array}{cccccc}
1 & 0 & 0 & \cdots & 0 & a_1^2 \\
0 & 1 & 0 & \cdots & 0 & a_2^2 \\
0 & 0 & 1 & \cdots & 0 & a_3^2 \\
\vdots & \vdots & \vdots & \ddots & \vdots & \vdots \\
0 & 0 & 0 & \cdots & 1 & a_n^2 \\
1 & 1 & 1 & \cdots & 1 & 0 \end{array} \right). \]

Since $A$ contains the $n \times n$ identity matrix as a submatrix, the square root rank of $A$ must be either $n$ or $n+1$.  If $\sqrt{A}$ is a Hadamard square root, then we may scale rows and columns of $\sqrt{A}$ by $-1$ and not affect the rank.  Thus, we may assume that the first $n$ columns of $\sqrt{A}$ are composed of zeros and positive ones.  With this assumption, we see immediately that there exists a Hadamard square root of rank $n$ if and only if the partition problem for $a_1,\ldots,a_n$ is satisfiable.
\end{proof}

\begin{remark}
Although the partition problem is NP-complete, it is only weakly NP-complete and admits a pseudo-polynomial time algorithm.  Thus, the above theorem does not rule out the existence of an algorithm for deciding $\sqrtrank$ that runs in time polynomial in the problem dimension and the \emph{magnitude} (not encoding length) of the matrix entries.  Furthermore, this embedding of the partition problem cannot hope to show that psd rank is NP-hard to compute.  To see this, consider the matrix $A$ corresponding to the partition problem with integers 5, 12, and 13:
\[ \left( \begin{array}{cccc}
1 & 0 & 0 & 25 \\
0 & 1 & 0 & 144 \\
0 & 0 & 1 & 169 \\
1 & 1 & 1 & 0 \end{array} \right).\]
This instance of the partition problem is not satisfiable, yet the matrix $A$ has a $3\times 3$ psd factorization.  Such a factorization is obtained by placing the matrices
\[ \left( \begin{array}{ccc}
1 & 0 & - \frac{5}{13} \\
0 & 1 & - \frac{12}{13} \\
- \frac{5}{13} & - \frac{12}{13} & 1 \end{array} \right),
\left( \begin{array}{ccc}
25 & 60 & 65 \\ 60 & 144 & 156 \\ 65 & 156 & 169 \end{array} \right) \]
on the fourth row and the fourth column, respectively, of $A$ and by placing the standard basis factorization of the identity in the first three rows and columns.
\end{remark}

\subsection{Lower bounds for psd rank}
\label{subsec:lower bounds for psd rank}
Lower bounding the psd rank has shown to be a difficult task.  In this section, we discuss the known lower bounding techniques and their limitations.

We say that two matrices of the same dimensions have the
same \textbf{support} if they share the same zero/nonzero pattern in
their entries.  Lower bounds based solely on the support of the matrix
have been shown to be quite powerful in the case of nonnegative rank
(see \cite{fiorini2013combbounds} for an overview).  In the case of
psd rank, their power is much more limited.  Given a matrix $M$, the
entry-wise square $M \circ M$ has the same support as $M$ and has psd
rank bounded above by $\rank (M)$ (Theorem~\ref{thm:firstprops}, part
(v)).  Thus, a purely support-based bound cannot produce a lower bound
that is higher than the rank of $M$.  This observation was extended by
Lee and Theis in \cite{leetheis2012support} as follows:

\begin{theorem} \cite[Theorem 1.1]{leetheis2012support}
Fix a support $Z$ and let $\mathcal{M}_Z$ be the set of all matrices sharing this support.  Then
\[ \min_{A \in \mathcal{M}_Z} \rank(A) = \min_{A \in \mathcal{M}_Z, A \geq 0} \rankpsd(A). \]
\end{theorem}

If a nonnegative matrix has the property that it achieves the minimum rank possible among all matrices sharing its support, then the rank is a lower bound to the psd rank.  In particular, slack matrices of polytopes have this property.  In \cite{gouveia2013polytopes}, the authors showed the following corollary and characterized those polytopes that achieve this lower bound in $\RR^2$ and $\RR^3$.

\begin{corollary} \label{cor:rank lower bound to psd rank}\cite[Proposition~3.2]{gouveia2013polytopes}
If $P$ is an $n$-dimensional polytope with slack matrix $S_P$, then $\rank(S_P) = n+1 \leq \rankpsd (S_P)$.
\end{corollary}

%
%

To obtain stronger lower bounds, we need to move past using only the support of a matrix.  The only known lower bounding techniques that are not purely support based rely on the quantifier elimination theory of Renegar \cite{renegar1992intro} as seen in \cite{gouveia2011lifts}.

We give a high level discussion of the idea behind this technique and then the result. For complete proofs, see
\cite{gouveia2011lifts}.  Suppose we are given a convex set $C \subset \RR^n$ that has a psd lift into $\PSD^k$, i.e. there exists a linear map $\pi$ and an affine subspace $\cL$ such that $C = \pi(\cL \cap \PSD^k )$.  Then $\cL \cap \PSD^k$ is a semialgebraic set where the bounding polynomials have degree at most $k$.

\begin{theorem} \cite{renegar2006hyperbolic} \label{thm:Renegar derivatives}
Let $Q = \{ z \in \RR^m \,:\, C + \sum z_i A_i \succeq 0 \}$ be a spectrahedron with
$E := C+\sum z'_i A_i \succ 0$ for some $z' \in Q$, and $C, A_i$ are symmetric matrices of size $k \times k$. Then $Q$ is a semialgebraic set described by $g^{(i)}(z) \geq 0$ for $i=1,\ldots,k$
where $g^{(0)}(z) := \textup{det} (C + \sum z_i A_i)$ and $g^{(i)}(z)$ is the $i$-th Renegar derivative of $g^{(0)}(z)$ in direction $E$.
\end{theorem}

The work of Renegar says that when we project this set, the degree and number of the resulting bounding polynomials of $C$ are bounded in $k$ and $n$.

\begin{theorem} \cite[Theorem 1.1]{renegar1992intro}
\label{thm:renegar}
Given a formula of the form
$$\exists \,\,y \in \RR^{m-n} \,:\, g_i(x,y) \geq 0 \,\,\,\,\forall \,i=1, \ldots, s$$
where $x \in \RR^n$ and $g_i \in \RR[x,y]$ are polynomials of degree at most $d$,
there exists a quantifier elimination method that produces a quantifier free formula of the form
\begin{equation} \label{quantifier-free formula}
\bigvee_{i=1}^{I} \bigwedge_{j=1}^{J_i} (h_{ij}(x) \,\Delta_{ij} \,0)
\end{equation}
where $h_{ij} \in \RR[x]$, $\Delta_{ij} \in \{>, \geq, =, \neq, \leq, < \}$ such that
$$I \leq (sd)^{{K}n(m-n)}, \,\,J_i \leq (sd)^{{K}(m-n)}$$
and the degree of  $h_{ij}$ is at most $(sd)^{{K}(m-n)}$, where
$K$ is a constant.
\end{theorem}

Multiplying all of these polynomials together, we obtain a single polynomial, whose degree is bounded in $k$ and $n$, that vanishes on the boundary of $C$.  Hence, if we know that every polynomial that vanishes on the boundary of $C$ must have very high degree, then we can say that $C$ does not have a $\PSD^k$-lift for small $k$.  The {\em Zariski closure} of the
boundary of $C$ is a hypersurface in $\RR^n$ since the boundary of $C$ has codimension one. We define the {\em degree of $C$} to be the minimal degree of a (nonzero) polynomial whose zero set is the
Zariski closure of the boundary of $C$. By construction, this polynomial vanishes on the boundary of $C$.

\begin{proposition} \cite[Proposition~6]{gouveia2011lifts} \label{prop:degree bound}
If $C \subseteq \RR^n$ is a full-dimensional convex semialgebraic set with a $\PSD^k$-lift,
then the degree of $C$ is at most $k^{O(k^2n)}$.
\end{proposition}

When $C$ is a polytope, the degree of $C$ is equal to the number of facets, i.e. the minimal polynomial vanishing on the boundary of $C$ is the product of all the linear polynomials determining the facets.  This lower bounds the psd rank of slack matrices of polytopes.

\begin{corollary} \cite[Corollary~4]{gouveia2011lifts} \label{cor:max facets with fixed psd rank}
If $C \subset \RR^n$ is a full-dimensional polytope whose slack matrix has psd rank $k$, then $C$ has
at most $k^{O(k^2n)}$ facets.
\end{corollary}

\begin{example} \label{ex:ngonsslack}
Corollary~\ref{cor:max facets with fixed psd rank} shows that as the number of facets in an $n$-dimensional polytope in $\RR^n$ increases, the psd rank of the slack matrix of the polytope has to increase. However, the rank of any such slack matrix stays fixed at $n+1$. For example, let $S_d$ be the slack matrix of a 
$d$-gon in the plane. Then by Corollary~\ref{cor:max facets with fixed psd rank},
$\rankpsd(S_d)$ grows to infinity as $d$ increases. As we have
seen before, however, $\rank(S_d)=3$ for all $d$. This provides a first example of a family of matrices with arbitrarily large gap between rank and psd rank.
\end{example}

For non-slack matrices, it can still be possible to apply this lower bound by viewing the matrix as a generalized slack matrix.

\begin{example} \label{ex:curve construction}
We now construct a matrix family that has the same zero pattern as the derangement matrices
and for which rank is three and psd rank grows arbitrarily large.
Let $C_d$ be a convex semialgebraic set in the plane whose bounding polynomial has degree $d$. By results of Scheiderer \cite{scheiderer2012representations},  $C_d$ has a $\PSD^r$-lift for some finite $r$, and suppose $k$ is the smallest such $r$. By Proposition~\ref{prop:degree bound}, $d \leq k^{O(k^2)}$.

Now pick $d^2+1$ distinct points on the boundary of $C_d$ and let $P$ be the convex hull of these points. Also, let $Q$ be the polyhedron whose facet inequalities are given by the tangent lines to $C_d$ at the vertices of $P$. Then the slack matrix of the pair $P,Q$, which was denoted as $S_{P,Q}$ in Section~\ref{subsec:geo interpretation}, is a nonnegative matrix with the same zero pattern as the derangement matrix. Call this nonnegative matrix $M_d$. By construction, the set $C_d$ is sandwiched between $P$ and $Q$ and its boundary passes through the vertices of $P$. If $C'$ is another convex semialgebraic set that is also sandwiched between $P$ and $Q$, then its boundary must also contain the vertices of $P$ because $P$ and $Q$ touch at these vertices. By Bezout's theorem, the degree of $C'$ must be at least $d+1$.
By Theorem~\ref{thm:psdlift}, the psd rank of $M_d$ is the smallest $k$ such that a slice of $\PSD^k$ projects
to a convex set nested between $P$ and $Q$. Since this smallest $k$ grows as the degree of the polynomial bounding the projected spectrahedron grows,
it must be that the psd rank of $M_d$ grows with $d$.

On the other hand,
$\rank (M_d) = 3$ since it is the slack matrix of a pair of polygons. Therefore, by choosing a family of convex semialgebraic sets
in the plane of increasing degree with the requirements specified above, one can obtain a family of nonnegative matrices of rank three and growing psd rank. For instance,
take $$C_{d=2t} := \{ (x,y) \in \RR^2 \,:\, x^{2t} + y^{2t} \leq 1 \}.$$
\end{example}

This quantifier elimination lower bound framework has proven useful for showing that certain families of matrices must have growing psd rank.  Its usefulness is limited, however, when considering the psd rank of a single matrix.  Other techniques have been developed to show matrices with high psd rank.  Bri\"{e}t et.~al. used a counting argument in \cite{briet2013} to show that most $0/1$-polytopes in $\RR^n$ have psd rank that is at least exponential in $n$.  Gouveia et.~al. produced a lower bound for generic polytopes (polytopes whose vertices are algebraically independent) in \cite{gouveia2013worst}, but again, these techniques are of limited usefulness when considering a single specific matrix.  An answer to the following problem would likely provide a new technique that is applicable to many open questions in this field.

\begin{problem}
Produce a $10 \times 10$ nonnegative matrix $M$ with integral entries such that $\rank (M) = 3$ and $\rankpsd (M) \geq 5$.
\end{problem}
\subsection{Comparisons between ranks}
We can now compare all the ranks seen so far. To keep track, we summarize the relationships in Table \ref{table:ranks}.

\begin{table}[h!]
 \caption{Relationships between various ranks}
 \begin{tabular}{c|||c|c|c|c}
                     & $\rank$     & $\rankplus$ & $\rankpsd$ & $\sqrtrank$ \\ \hline
 $\rank$        & $=$           & $\ll$~(\ref{ex:euclidean distance})            & $\ll$ (\ref{ex:ngonsslack}, \ref{ex:curve construction}) , $>$ ~(\ref{ex:derangement})       & $\ll$ (\ref{ex:prime corners matrix}), $ >$ ~(\ref{ex:euclidean distance}) \\
 $\rankplus$ & $\gg$         & $=$                                                      & $\gg$ (\ref{ex:euclidean distance})           & $\ll$ (\ref{ex:prime corners matrix}), $\gg$ (\ref{ex:euclidean distance})\\
 $\rankpsd$  & $< , \gg$    & $\ll$          & $=$             & $\ll$ (\ref{ex:prime corners matrix})  \\
 $\sqrtrank$  & $< , \gg$    & $\ll$ , $\gg$   & $\gg$          & $=$
\end{tabular}
\label{table:ranks}
\end{table}

The symbol $\ll$ refers to the rank indexing the row being arbitrarily
smaller than the one indexing the column (i.e. there does not exist a
function of the row rank that upper bounds the column rank).  The
symbol $<$ indicates that the rank on the row may be smaller than the
rank on the column, but the gap cannot not be arbitrarily large.  For
example, the entry in the $(1,2)$-position says that rank may be
arbitrarily smaller than nonnegative rank, but never larger.  The
$(1,3)$-entry says that rank may be smaller or larger than psd rank.
The gap in the first case may be arbitrarily large, but the gap in the
second case is controlled (see \eqref{eq:easy rank psd rank}). The
numbers in the table refer to examples exhibiting the relationship.

\begin{example} [{\bf Euclidean distance matrices}] \label{ex:euclidean distance}
Consider the $n \times n$ {\em Euclidean distance matrix} $M_n$ whose $(i,j)$-entry is
$(i-j)^2$. The rank of $M_n$ is three for all $n$ since $M_n = A_n^T B_n$ where
column $i$ of $A_n$ is $(i^2 , -2i , 1)$ and column
$j$ of $B_n$ is $(1 , j , j^2)$.

The square root rank and psd rank of $M_n$ are two for all $n$ since the matrix with $(i,j)$-entry equal to $i-j$ has usual rank two.
So for all $n$, the matrix $M_n$ has constant size rank, psd rank, and square root rank.

Now we show that $M_n$ has growing nonnegative rank.  Suppose $M_n$ has a $\RR^k_+$-factorization. Then there exists $a_1, \ldots, a_n, b_1,
\ldots, b_n \in \RR^k_+$ such that $\langle a_i, b_j \rangle \neq 0$
for all $i \neq j$. Notice that if $\supp(b_j)
\subseteq \supp(b_i)$ then $\langle a_i , b_i \rangle =0$ implies
$\langle a_i, b_j \rangle = 0$, and hence, all the $b_i$'s (and also
all the $a_i$'s) must have supports that are pairwise incomparable. Since there are at most 
$2^k$ possibilities for these supports, $n \leq 2^k$, or equivalently, $k \geq \log_2 n$. Therefore we get that 
$\rankplus(M_n) \geq \log_2 n$.

In \cite{hrubes2012edm}, Hrube\v{s} exhibited a nonnegative factorization to show that the nonnegative rank of this family is actually $\Theta(\log_2 n)$.
\end{example}

\begin{example} [{\bf Prime matrices}] \label{ex:prime corners matrix}
Let $n_1,n_2,n_3,\ldots$ be an increasing sequence of positive integers such that $2n_k - 1$ is prime for each $k$.  Let $\mathcal{P}_k$ denote the set of all primes strictly less than $2n_k-1$.
Define a $k \times k$ matrix $Q^k$ such that $Q^k_{ij} = n_i + n_j - 1$.  Then $Q^k$ has usual rank two for all $k$.  Consequently, by Proposition~\ref{prop:rank1rank2}, the nonnegative rank and psd rank of $Q^k$ are also two for all $k$.
For example, suppose our sequence has the form $2,3,4,6,\ldots$  Then $Q^1,\ldots,Q^4$ will have the form:
\[ \left( \begin{array}{c} 3 \end{array} \right),
\left( \begin{array}{cc} 3 & 4 \\ 4 & 5 \end{array} \right),
\left( \begin{array}{ccc} 3 & 4 & 5 \\ 4 & 5 & 6 \\ 5 & 6 & 7 \end{array} \right),
\left( \begin{array}{cccc} 3 & 4 & 5 & 7\\ 4 & 5 & 6 & 8\\ 5 & 6 & 7 & 9 \\ 7 & 8 & 9 & 11 \end{array} \right) . \]
Note that the top left block of each $Q^k$ is $Q^{k-1}$ and that the diagonal entries of $Q^k$ are the increasing sequence of primes, $2n_1-1,2n_2-1,\ldots$

We will prove by induction that $Q^k$ has full square root rank for each $k$.  The base case is clear so assume that $Q^{k-1}$ has full square root rank, i.e. every possible Hadamard square root of $Q^{k-1}$ has rank equal to $k-1$.  Fix a Hadamard square root of $Q^k$ and let $M$ be the matrix equal to this square root in every entry except $M_{kk}$.  Let $M_{kk}$ be the variable $x$. The determinant of $M$ has the form $\alpha x + \beta$ where $\alpha$ and $\beta$ are in the extension field $\QQ(\sqrt{\mathcal{P}_k})$.  By properties of the determinant, $\alpha$ is equal to the determinant of the top left $(k-1) \times (k-1)$ block.  Thus by the induction assumption, $\alpha$ is nonzero.  Hence, any $x$ making the determinant zero must also lie in $\QQ(\sqrt{\mathcal{P}_k})$.  However, our square root of $Q^k$ must have $x = \pm \sqrt{2n_k-1}$.  Thus, the square root of $Q^k$ must have full rank.
\end{example}

The remaining relationship shown in Table~\ref{table:ranks} that we have not discussed is $\rank > \sqrtrank$.  In the example after Definition~\ref{def:sqrt rank}, we saw that rank can be larger than square root rank.  The possible gap is controlled, however, since square root rank is an upper bound to psd rank and the gap between rank and psd rank is controlled.

\section{Properties of factors}
\label{sec:properties}
The matrices used as factors in a positive semidefinite factorization
can sometimes be chosen to satisfy specific constraints. In this section we explore the ranks and norms of the factors used in a psd factorization.

\subsection{Rank of factors}

In \cite{leetheis2012support} it has been shown that we can always pick the factors of a factorization to have some bounded rank, depending only on the size of the matrix:
\begin{proposition}\cite[Lemma~4.5]{leetheis2012support}
If a $p \times q$ nonnegative matrix $M$ has a $\PSD^k$ factorization, then it has one using factors of rank at most $\sqrt{8q+1}/2$ for the rows and
at most $\sqrt{8p+1}/2$ for the columns.
\end{proposition}
The reason is simple. If we fix the factors corresponding to rows, the
set of valid factors for any given column is the feasible set of a
semidefinite program, and standard results from convex optimization
guarantee the existence of a solution with bounded rank
(\cite{pataki1998rank},\cite{barvinok2001remark}). Fixing these column
factors and repeating the process over the rows we get the result.

We are particularly interested in knowing for which cases can the factors be chosen to have rank one. The answer turns out to be given by the $\sqrtrank(M)$ (Definition \ref{def:sqrt rank}).
\begin{proposition} \label{prop:square root rank and rank one factors}
The square root rank of $M$, $\sqrtrank(M)$, is precisely the smallest size of a psd factorization of $M$ comprised solely of rank one factors.
\end{proposition}
\begin{proof}
As seen in the proof of Theorem~\ref{thm:firstprops} (v), if $M=N \hadprod N$ and $N$ has rank $k$, then we can take a rank factorization of $N$, $N=A^TB$, and use it to create the matrices
$A_i = a_i a_i^T$ and $B_j =b_j b_j^T$, where $a_i$ and $b_j$ range over the columns of $A$ and $B$ respectively. These matrices $A_i$ and $B_j$ have rank one and, by construction, form a $\PSD^k$ factorization of $M$.

To prove the reverse implication just note that if $v_iv_i^T$ and $w_jw_j^T$ form a $\PSD^k$ factorization of $M$ then setting $V$ and $W$ to be the matrices whose columns are the $v_i$ and the $w_j$ respectively, we
can obtain a matrix $V^TW$ that has rank $k$ and is a Hadamard square root of $M$.
\end{proof}

In general, we expect the gap between $\sqrtrank(M)$ and $\rankpsd(M)$
to be arbitrarily high, as illustrated in Example~\ref{ex:prime
corners matrix}, but good knowledge of rank one factorizations can
potentially provide some insight into general factorizations.

\begin{remark}\label{rem:rank1projection}
Let $M$ be a $p \times q$ nonnegative matrix with $\rankpsd(M)=k$. Suppose $A_i$, $i=1,\ldots,p$ and $B_j$, $j=1,\ldots,q$ form a $\PSD^k$ factorization of $M$. Each $A_i$ can be written as $\sum_{l=1}^k v_{i,l} v_{i,l}^T$, and each $B_j$ as $\sum_{l=1}^k w_{j,l} w_{j,l}^T$. Define $N$ as the matrix indexed by $\{1, \ldots, p\} \times \{1, \ldots, k\}$ and $\{1, \ldots, q\} \times \{1, \ldots, k\}$ whose entry $N_{\{i,s\},\{j,r\}}$ is given by $(v_{i,s}^T w_{j,r})^2$. 

Then, $N \in \RR_+^{pk \times qk}$ consists of $p\times q$ blocks of size $k \times k$ and $\sqrtrank(N)=k$. Furthermore, summing all the entries of block $(i,j)$ of $N$ gives us entry $(i,j)$ of $M$. 
\end{remark}

This remark allows us to transfer properties from rank one factorizations to general factorizations, an example can be seen at the end of the next subsection.

\subsection{Norms of factors}

Besides rank, another useful quantity to control in the factors is
their size. In Section \ref{sec:definitions} we used a lemma showing that the factors of a semidefinite factorization can be rescaled to have small
trace. This was instrumental in establishing the lower semicontinuity
of psd rank. We restate the lemma here and provide a proof:
\begin{lemma}
\label{lem:simplescaling2}
Let $M \in \RR^{p\times q}_+$ and assume that $M$ has a psd factorization of size $k$. Then $M$ admits a psd factorization $M_{ij} =\langle A_i, B_j \rangle$ of size $k$ where the factors satisfy $\trace(A_i)\leq k$ and $\trace(B_j) = \sum_{i=1}^p M_{ij}$.
\end{lemma}
\begin{proof}
Let $M_{ij} = \langle \widehat{A}_i, \widehat{B}_j \rangle$ be an arbitrary psd factorization of $M$ of size $k$. Let $S = \sum_{i=1}^p \widehat{A}_i \in \S^k_+$. Define $A_i = S^{-1/2} \widehat{A}_i S^{-1/2}$ and let $B_j = S^{1/2} \widehat{B}_j S^{1/2}$. Note that $\langle A_i, B_j \rangle = \langle \widehat{A}_i, \widehat{B}_j \rangle = M_{ij}$ and so $A_1,\dots,A_p,B_1,\dots,B_q$ give a valid psd factorization of $M$ of size $k$. Observe that by construction we have $\sum_{i=1}^p A_i = I$ (where $I$ is the $k\times k$ identity matrix), thus necessarily each $A_i$ satisfies $A_i \preceq I$ and thus $\trace(A_i) \leq k$. Also we have for any $j=1,\dots,q$, $\trace(B_j) = \trace((\sum_{i=1}^p A_i)B_j) = \sum_{i=1}^p \trace(A_i B_j) = \sum_{i=1}^p M_{ij}$ as desired.
\end{proof}

Note that this is equivalent to saying that we can choose $A_i$ and $B_j$ all with trace bounded by $\sqrt{k \|M\|_{1,1}}$ where $\|M\|_{1,1}$ is the matrix norm induced by the $1$-norm in $\RR^k$. 
This looks very similar to another rescaling result that has proven very useful, the rescaling result in~\cite{briet2013}, which was used to show that
there are $0/1$-polytopes with only exponential-sized semidefinite
representations.  The result states that if a matrix $M$ has psd rank
$k$, then it has a semidefinite factorization where each factor has
largest eigenvalue less than or equal to $\sqrt{k \|M\|_{\infty}}$, where 
$\|M\|_{\infty}$ is the maximum absolute value of an entry in $M$.

In the remainder of this section we present a new, simplified proof of this
fact. As in~\cite{briet2013}, the main tool we need is a
version of \emph{John's ellipsoid theorem}.

\begin{theorem}[John's Theorem] Let $C$ be a full dimensional convex set in $\RR^n$ and let $T:\RR^n\rightarrow \RR^n$ be a linear map such that the image of the unit ball $B^n$
under $T$ is the unique minimum volume ellipsoid $E$ containing $C \cup -C$. Then $$\frac{1}{n} TT^T \in \conv(\{vv^T : v \in \textup{Boundary}(C) \cap \textup{Boundary}(E)\}).$$
\end{theorem}

A simple consequence of John's Theorem has to do with scalability of inner product realizations.

\begin{corollary} \label{cor:inner_rescaling}
Suppose $U, V \subseteq \RR^n$ are bounded and each span $\RR^n$, and $\Delta= \max_{u \in U, v\in V}\left|\left<u,v\right>\right|$. Then there exists a linear operator $L:\RR^n\rightarrow \RR^n$ such that $\max_{u \in U} \|Lu\|_2$ and $\max_{v \in V} \|(L^{-1})^Tv\|_2$ are both less than or equal to $n^{1/4}\sqrt{\Delta}$.
\end{corollary}
\begin{proof}
Consider $C=\conv(U)$, and $T:\RR^n\rightarrow \RR^n$ such that $T(B^n)$ is the minimum volume ellipsoid containing $C \cup -C$. For all $u \in U$, $\|T^{-1}(u)\|_2 \leq 1$ by construction. Furthermore by John's theorem for any $v\in V$
$$\|T^T(v)\|_2^2=v^TTT^Tv=v^T n (\sum_i \lambda_i u_iu_i^T) v \leq n \sum_i \lambda_i \left< u_i,v\right>^2$$
with $\lambda_i \geq 0$ and $\sum \lambda_i = 1$. But this implies $\|T^T(v)\|_2^2 \leq n \Delta^2$ and therefore $\|T^T(v)\|_2 \leq \sqrt{n} \Delta$. By making $L=n^{1/4} \sqrt{\Delta} {T}^{-1}$ we get the intended result.
\end{proof}

This immediately gives us the fact that the usual matrix factorization is scalable.

\begin{corollary}\label{cor:usual_rank_scalability}
If $M\in \RR^{p\times q}$ has rank $k$, then there exist $A \in \RR^{k \times p}$, $B \in \RR^{k \times q}$ such that $M=A^TB$ and
the maximum $2$-norm of a column of $A$ or $B$ is at most $k^{1/4} \sqrt{\|M\|_\infty}$.
\end{corollary}

\begin{proof}
Start with any rank factorization $M=U^TV$ and apply Corollary \ref{cor:inner_rescaling} to the columns of $U$ and $V$. Then $A=LU$ and $B=(L^{-1})^TV$ have the intended properties.
\end{proof}

As mentioned in the introduction, factorizations where the factors
have small norm have been studied in different contexts, particularly
in Banach space theory. For example, \cite[Lemma
4.2]{LinialMendelsonSchechtmanShraibman} is closely related to
Corollary~\ref{cor:usual_rank_scalability}. The scalability of psd
factorizations also follows readily.

\begin{corollary} \label{cor:eigbound}
If $M\in \RR_+^{p\times q}$ has psd rank $k$, then there exist $A_1, \ldots, A_p, B_1, \ldots, B_q \in \PSD^k$ such that $M_{i,j}=\left<A_i,B_j\right>$ and the largest eigenvalue
of $A_i$ and $B_j$ is bounded above by $\sqrt{k \|M\|_{\infty}}$.
\end{corollary}

\begin{proof}
Start with a factorization $A'_1, \ldots, A'_p, B'_1, \ldots, B'_q \in \PSD^k$, and let
$$U=\{ u \in \RR^k \ : \ A'_i - uu^T \succeq  0 \textrm{ for some } i\},$$
while
$$V=\{ v \in \RR^k \ : \ B'_j - vv^T \succeq 0 \textrm{ for some } j\}.$$
For $u \in U$ and $v \in V$, $\left< u,v \right> \leq \sqrt{\|M\|_{\infty}}$, since 
$$\left< u,v \right>^2 =  \left< uu^T,vv^T \right> \leq \left<A_i,B_j\right> = M_{i,j} \leq \|M\|_{\infty}.$$
Applying Corollary \ref{cor:inner_rescaling} to $U$ and $V$ we get a linear operator $L$ such that $(L^{-1})^Tv$
and $Lu$ have norm at most $\sqrt[4]{k \|M\|_{\infty}}$ for all $u \in U$ and $v \in V$. Let ${A_i}=LA'_iL^T$ and  $B_j=(L^{-1})^TB'_jL^{-1}$ then
$$LU=\{ u \in \RR^k \ : \ A_i - uu^T \succeq  0 \textrm{ for some } i\},$$
and
$$(L^{-1})^TV=\{ v \in \RR^k \ : \ B_j - vv^T \succeq 0 \textrm{ for some } j\}.$$
But note that if $\lambda$ is an eigenvalue of a psd matrix $A$ and $x$ a corresponding eigenvector with $\|x\|^2_2 = \lambda$, $A-xx^T \succeq 0$. This can be seen by considering an
eigenvector decomposition $A = \sum \lambda_i u_i u_i^T$ where $u_i=x/\|x\|_2$. In particular this implies that all eigenvalues of matrices $A_i$ can be seen as the square
of the norm of a vector in $LU$, and similarly with the $B_j$ and $LV$. Hence the maximum eigenvalues in each case are at most $\sqrt{k \|M\|_{\infty}}$ as intended.
\end{proof}

Note that if we are only interested in getting a bound of
$k^{5/4}\sqrt{\|M\|_{\infty}}$ (which is already enough for the
application in \cite{briet2013}) we can derive it directly from
Corollary~\ref{cor:usual_rank_scalability}, together with
Remark~\ref{rem:rank1projection}, illustrating that properties valid
for rank one factors can sometimes be extended to general
factorizations.

\begin{remark}
It is worth noting that while Lemma~\ref{lem:simplescaling2} and
Corollary~\ref{cor:eigbound} look similar, the bounds they provide are
in general not comparable. In general, if $M$ is dense, $\|M\|_{1,1}$
is expected to be much larger than $\|M\|_{\infty}$ so if the psd rank
is low compared to the number of rows of $M$, we expect the bound from
Corollary~\ref{cor:eigbound} to be significantly smaller. For sparse
matrices, the same is not true: when applied to the identity matrix
for example, Corollary~\ref{cor:eigbound} can only guarantee factors
of largest eigenvalue at most $\sqrt{k}$, hence the trace is at most
$k\sqrt{k}$, a worse guarantee than that obtained directly from
Lemma~\ref{lem:simplescaling2}.
\end{remark}

\section{Space of factorizations}
\label{sec:space of factorizations}
In this section, we fix a nonnegative matrix $M$ and consider the set of all valid psd factorizations of $M$ as a topological space.  In the special case where $\rank (M) = \binom{\rankpsd (M) + 1}{2}$, we show in Proposition~\ref{prop:space of factorizations} and Corollary~\ref{cor:space of factorizations} that this topological space is closely related to the space of all linear images of the psd cone that nest between two polyhedral cones coming from $M$.  An extension of this result to general $M$ is not possible, as seen in Example~\ref{ex:hexagon space}.  In Examples~\ref{ex:diff psd factorizations} and \ref{ex:full rank psd factors}, we use this machinery to construct psd factorizations from the linear embedding of the psd cone.  Finally in Proposition~\ref{prop:32 connected}, we show that for rank three matrices with psd rank two, the space of psd factorizations is connected.  This contrasts with the nonnegative rank case where it is known that the space of nonnegative factorizations can be disconnected for rank three matrices with nonnegative rank three \cite{straten2003sandwich}.

For this section, let $M \in \RR_+^{p \times q}$ be a nonnegative matrix with psd rank $k$.  As before, we define a \emph{psd factorization} to be a set of matrices $\left( A_1,\ldots,A_p, B_1,\ldots,B_q \right) \in \left( \CalS^k \right)^{p+q}$ such that each of the component matrices is psd and $M_{ij} = \langle A_i, B_j \rangle$ for each entry in $M$.  We define the set of all such psd factorizations to be the \emph{space of psd factorizations} associated to $M$ and denote it by $\mathcal{S}\mathcal{F}(M)$.  Note that this definition only considers matrices whose size is equal to the psd rank of $M$.  

As a warm-up, it is straightforward to see that $\mathcal{S}\mathcal{F}(M)$ is closed and infinite.  To see that $\mathcal{S}\mathcal{F}(M)$ is infinite, simply note that for any psd factorization $\left( A_1,\ldots,B_q \right)$ and 
any matrix $L \in GL(k)$ (the group of invertible $k \times k$ matrices),
the matrices $\left( L^T A_1 L,\ldots,L^{-1} B_q L^{-T} \right)$ also form a psd factorization of $M$.  We refer to the set of all such psd factorizations as the \emph{orbit} of $\left( A_1,\ldots,B_q \right)$ in $\mathcal{S}\mathcal{F}(M)$.  In some cases the entire space of psd factorizations is equivalent to a single orbit.


\begin{example}\label{ex:orbit of factorization}
Let $M$ be the $3 \times 3$ derangement matrix from Example~\ref{ex:3x3derangement} (i.e. $M_{ii} = 0$ and $M_{ij} = 1$ for $i \neq j $).  
This matrix has usual rank three and psd rank two as shown by the factorization

$$ \left[ \begin{array}{cc} 1 & 0 \\ 0 & 0 \end{array} \right], \left[ \begin{array}{cc} 0 & 0 \\ 0 & 1 \end{array} \right], \left[ \begin{array}{cc} 1 & -1 \\ -1 & 1 \end{array} \right], \left[ \begin{array}{cc} 0 & 0 \\ 0 & 1 \end{array} \right],\left[ \begin{array}{cc} 1 & 0 \\ 0 & 0 \end{array} \right], \left[ \begin{array}{cc} 1 & 1 \\ 1 & 1 \end{array} \right]. $$

Let $\XX$ denote an arbitrary psd factorization of $M$.  The zero pattern of $M$ implies that the matrices composing $\XX$ must all be rank one.  Now it is straightforward to see that there exists an invertible matrix such that conjugation by this matrix will send $\XX$ to the explicit factorization above.  Hence, $\SF(M)$ is composed of a single orbit.

\end{example}

The next proposition gives a geometric picture of the space of factorizations in the special case where $\rankpsd (M) = k$ and $\rank (M) = \binom{k+1}{2}$.  Before proceeding, we present a homogenized version of the geometric interpretation of psd rank that was given in Section~\ref{subsec:geo interpretation}, which will be easier to work with in this section.  First, we homogenize Definition~\ref{def:slack}:
\begin{definition}\label{def:cone slack}
Let $P$ and $Q$ be polyhedral cones with $P\subset Q \subset \RR^n$. Let $x_1,\dots,x_v$ be the extreme rays of $P$ and let $a_j^Tx \geq 0$, $(j=1,\dots,f)$ be the inequalities defining $Q$. Then the slack matrix of the pair $P,Q$, denoted $S_{P,Q}$, is the nonnegative $v\times f$ matrix whose $(i,j)$th entry is $a_j^T x_i$.
\end{definition}
By taking a rank factorization, it is easy to see that any nonnegative matrix $M$ can be viewed as $S_{P,Q}$ for some polyhedral cones $P$ and $Q$ whose dimension is equal to $\rank (M)$.
Furthermore, Theorem~\ref{thm:psdlift} extends straightforwardly to this conic setting: $\rankpsd \left(S_{P,Q}\right)$ is the smallest integer $k$ for which there exists a subspace $\cL$ and a linear map $\pi$ such that $P \subset \pi(\S^k_+ \cap \cL) \subset Q$.

In the special case where the cones $P$ and $Q$ come from a matrix $M$ with $\rankpsd (M) = k$ and $\rank (M) = \binom{k+1}{2}$, we can count dimensions to see that the map $\pi$ is invertible and the subspace $\cL$ is all of $\CalS^k$.  If we define $\Delta_k(P,Q)$ to be the space of all linear maps $\pi : \S^k \rightarrow \RR^{\binom{k+1}{2}}$ such that $P \subset \pi(\PSD^k) \subset Q$, then $\Delta_k(P,Q)$ is nonempty with $M$, $P$, and $Q$ as above.  We can actually say much more about $\Delta_k(P,Q)$ in this special case.



\begin{proposition}\label{prop:space of factorizations}
Let $M \in \RR_+^{p \times q}$ with $\rank (M) = \binom{k+1}{2}$ and $\rankpsd (M) = k$.  Fix a rank factorization of $M = UV$ where $u_i$ is the $i$th row of $U$ and $v_j$ is the $j$th column of $V$.  Let $P = \cone(u_1,\ldots,u_p)$ and $Q = \left\{ x \; | \; v_j^T x \geq 0 \textup{ for all } j \right\}$ be the cones generated by this rank factorization so that $M = S_{P,Q}$.  Then $\SF(M)$ is homeomorphic to $\Delta_k(P,Q)$.
\end{proposition}

\begin{proof}
Suppose $\left( A_1,\ldots,B_q \right)$ is a psd factorization of $M$.  The set $\left(A_1,\ldots, A_p\right)$ spans $\S^k$ (else we could find a lower dimensional rank factorization of $M$), so we can define a linear map $\pi$ by making $\pi(A_i) = u_i$.  This map is well-defined since if $\sum_i{\alpha_i A_i}$ and $\sum_j{\beta_j A_j}$ are two representations of the same matrix in $\S^k$, then we have that $\left( \sum_i{\alpha_i u_i} - \sum_j{\beta_j u_j}\right)^T V = 0$.  Since $V$ has full row rank, this implies that 
$\sum_i{\alpha_i u_i} = \sum_j{\beta_j u_j}$.  By the definition of $\pi$, it is immediate that $P \subset \pi(\PSD^k)$.  Since $\pi$ has the property that $\langle \pi(L), v_j \rangle = \langle L, B_j \rangle$ for each $j$, we also have that $\pi(\PSD^k) \subset Q$.  Thus, we have defined a map from 
$\SF(M)$ to $\Delta_k(P,Q)$.

Next, suppose that we have $\pi \in \Delta_k(P,Q)$.  Define $A_i = \pi^{-1}(u_i)$ and $B_j = \pi^*(v_j)$ where $\pi^*$ is the adjoint map.  Then $A_i \in \PSD^k$ and $B_j \in (\PSD^k)^* = \PSD^k$ and these matrices form a psd factorization of $M$.  This map is the inverse of the one defined above and both of the maps are continuous.  Hence, the spaces are homeomorphic.
\end{proof} 

Both of the spaces in the previous proposition permit a natural action by $GL(k)$.  The action on $\SF(M)$ was mentioned above when we discussed the orbits of $\SF(M)$.  The action on $\Delta_k(P,Q)$ takes the form $g \cdot \pi(L) = \pi(g L g^T)$, i.e. we compose the map $\pi$ with an automorphism of the psd cone.  
The homeomorphism in the previous proposition respects these group actions so we can descend to the quotient to see the following.

\begin{corollary}\label{cor:space of factorizations}
Under the same assumptions as the previous proposition, the spaces $\SF(M) / GL(k)$ and $\Delta_k(P,Q) / GL(k)$ are homeomorphic.  Furthermore, $\Delta_k(P,Q) / GL(k)$ is homeomorphic to the space of all linear images $C$ of $\PSD^k$ such that $P \subset C \subset Q$.
\end{corollary}

\begin{proof}
The first statement is shown by descending to the quotient as discussed prior to the corollary.  The second homeomorphism is just given by $[\pi] \mapsto \pi(\PSD^k)$.  It is straightforward to check that this map is a well-defined homeomorphism.
\end{proof}


The next example shows that the conclusion of Corollary~\ref{cor:space of factorizations} cannot hold for general $M$.

\begin{example} \label{ex:hexagon space}
Let $M$ be the slack matrix of the regular hexagon, i.e. $M$ is the $6 \times 6$ circulant matrix defined by the vector $(0,1,2,2,1,0)$.  It was shown in \cite{gouveia2013polytopes} that $M$ has rank three, psd rank four, and at least two distinct factorization orbits (since there exists a factorization consisting entirely of rank one matrices and another factorization using both rank one and rank two matrices).  Since this matrix is a slack matrix of a polytope, however, the cones $P$ and $Q$ must be equal and the only image nested between them must be $P$ itself.  Hence, there cannot exist a bijection between factorization orbits and images nested between $P$ and $Q$.
\end{example}

In the next examples, we apply our machinery to matrices with rank three and psd rank two.  

\begin{remark}
In light of Corollary~\ref{cor:space of factorizations}, we can gain new insight on Example~\ref{ex:orbit of factorization}.  By taking the trivial rank factorization of the $3 \times 3$ derangement matrix, we obtain the cones 
$$ P = \cone\left( (0,1,1),(1,0,1),(1,1,0) \right) \subset \RR^3_+ = Q . $$  
Dehomogenizing these cones gives us the two triangles seen in Figure~\ref{fig:derangement space}.  In this dehomogenized picture, linear images of the psd cone correspond to ellipses and it is straightforward to see that there is a unique ellipse that fits between the two triangles.  Hence, $\Delta_2(P,Q) / GL(2)$ consists of a single point and by the corollary, the space of psd factorizations is composed of a single orbit.
\end{remark}

\begin{figure}[ht]
  \includegraphics[width=4.5cm]{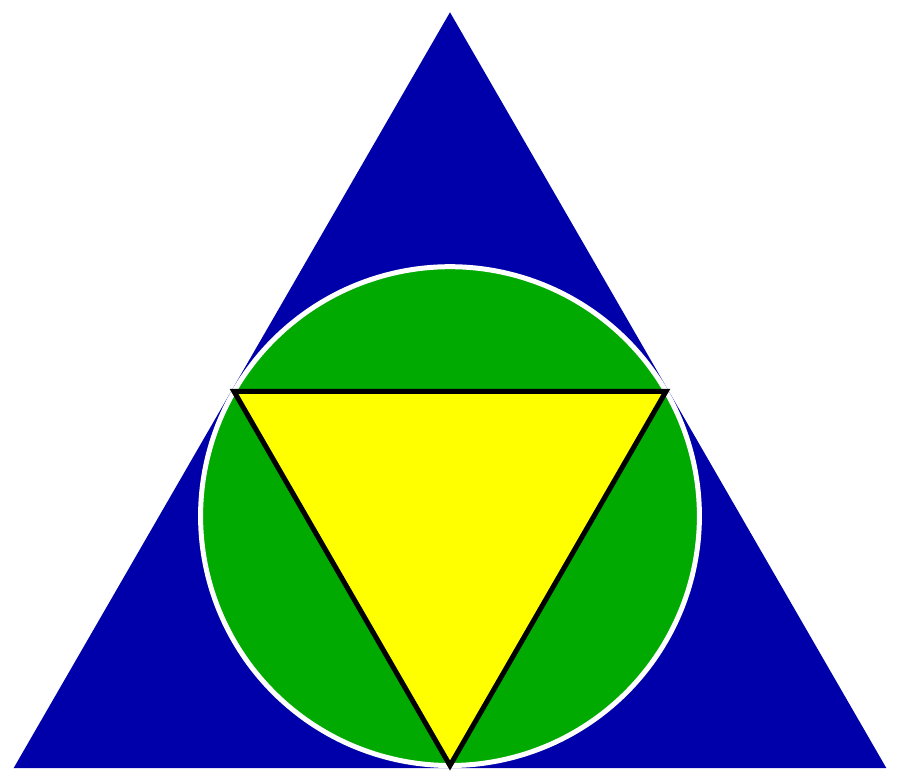}
  \caption{A space of psd factorizations consisting of a single orbit: The yellow and blue triangles correspond to the dehomogenized cones coming from a rank factorization of the $3 \times 3$ derangement matrix.  The green circle is the unique ellipse that can be nested between the two triangles. }
  \label{fig:derangement space}
\end{figure}

We now show how to apply Proposition~\ref{prop:space of factorizations} to find different psd factorizations of a matrix.

\begin{example}\label{ex:diff psd factorizations}
In this example, we consider the following matrix $M$ of rank three (shown along with a rank factorization):
\[ \left( \begin{array}{cccc} 3 & 3 & 1 & 1 \\ 1 & 3 & 3 & 1 \\ 1 & 1 & 3 & 3 \\ 3 & 1 & 1 & 3 \end{array} \right) = 
\left( \begin{array}{ccc} 1 & 1 & 1 \\ 1 & 1 & -1 \\ 1 & -1 & -1 \\ 1 & -1 & 1 \end{array} \right)
\left( \begin{array}{cccc} 2 & 2 & 2 & 2 \\ 0 & 1 & 0 & -1 \\ 1 & 0 & -1 & 0 \end{array} \right) . \]

By forming the cones $P$ and $Q$ corresponding to this rank factorization and then dehomogenizing through the plane $\left\{ (1,x_2,x_3) \right\}$, we see that $P$ corresponds to the square centered at the origin of side length two and that $Q$ corresponds to the same square scaled by a factor of two (see Figure~\ref{fig:diff psd factorizations}).  Now any linear image of $\PSD^2$ corresponds to an ellipse in this dehomogenized picture.  So to get a psd factorization of $M$, we just need to pick an ellipse, figure out a linear image of $\PSD^2$ that corresponds to it, and apply the homeomorphism discussed in Proposition~\ref{prop:space of factorizations}.

For the circle centered at the origin with radius $\sqrt{2}$, we get the following (where $\alpha = \frac{1}{\sqrt{2}}$):
\[ \left( \begin{array}{cc} 1+\alpha & \alpha \\ \alpha & 1-\alpha \end{array} \right),
\left( \begin{array}{cc} 1-\alpha & \alpha \\ \alpha & 1+\alpha \end{array} \right),
\left( \begin{array}{cc} 1-\alpha & -\alpha \\ -\alpha & 1+\alpha \end{array} \right),
\left( \begin{array}{cc} 1+\alpha & -\alpha \\ -\alpha & 1-\alpha \end{array} \right), \]

\[ \left( \begin{array}{cc} 1+\alpha & 0 \\ 0 & 1-\alpha \end{array} \right),
\left( \begin{array}{cc} 1 & \alpha \\ \alpha & 1 \end{array} \right),
\left( \begin{array}{cc} 1-\alpha & 0 \\ 0 & 1+\alpha \end{array} \right),
\left( \begin{array}{cc} 1 & -\alpha \\ -\alpha & 1 \end{array} \right) . \]

For the ellipse centered at the origin with horizontal axis of length four and vertical axis of length three, we get the factorization:
\[ \left( \begin{array}{cc} 5/3 & 1/2 \\ 1/2 & 1/3 \end{array} \right),
\left( \begin{array}{cc} 1/3 & 1/2 \\ 1/2 & 5/3 \end{array} \right),
\left( \begin{array}{cc} 1/3 & -1/2 \\ -1/2 & 5/3 \end{array} \right),
\left( \begin{array}{cc} 5/3 & -1/2 \\ -1/2 & 1/3 \end{array} \right), \]

\[ \left( \begin{array}{cc} 7/4 & 0 \\ 0 & 1/4 \end{array} \right),
\left( \begin{array}{cc} 1 & 1 \\ 1 & 1 \end{array} \right),
\left( \begin{array}{cc} 1/4 & 0 \\ 0 & 7/4 \end{array} \right),
\left( \begin{array}{cc} 1 & -1 \\ -1 & 1 \end{array} \right) . \]

It is interesting to note how the ranks of the factors change depending on whether the ellipse contacts the vertices of $P$ or the facets of $Q$.  For example, in the second factorization, the matrices corresponding to the columns are rank one exactly when the corresponding facet of the outer square is tight to the ellipse.  Of course, this is not a coincidence, but due to how we construct the psd factorization once we know the linear embedding of the psd cone. 
\end{example}

\begin{figure}[ht]
  \includegraphics[width=4.5cm]{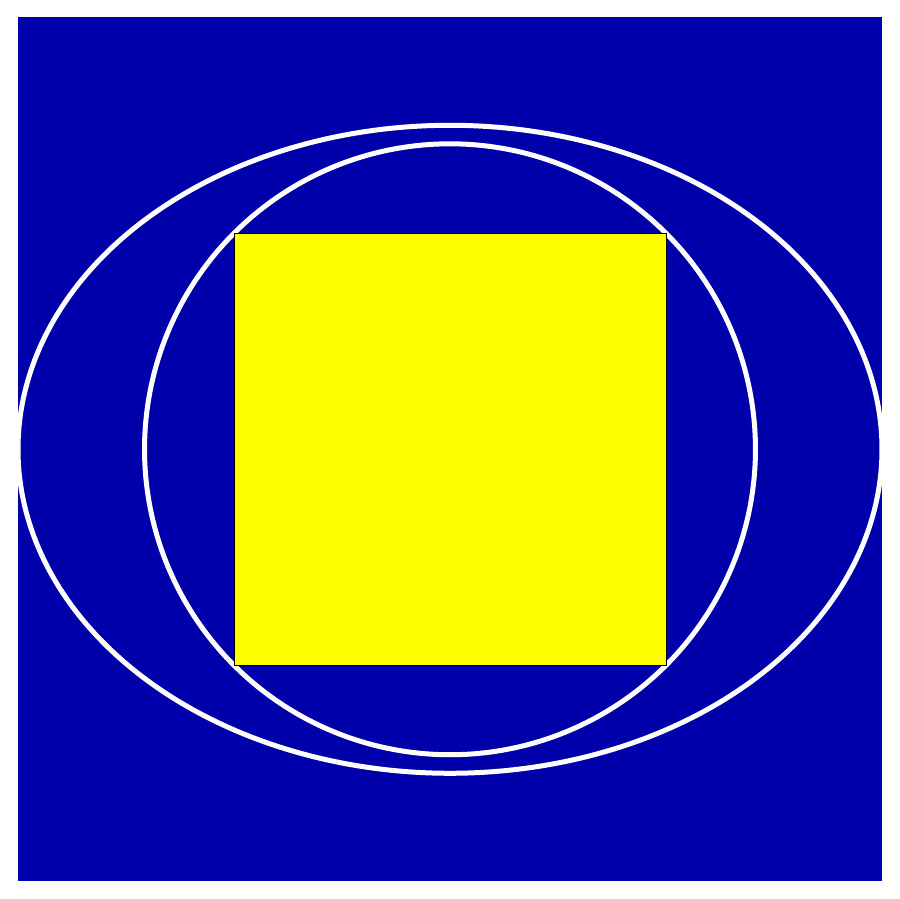}
  \caption{This shows the situation described in Example~\ref{ex:diff psd factorizations}.  The inner and outer squares correspond to the dehomogenized cones $P$ and $Q$ and any ellipse nested between the two squares corresponds to an orbit of psd factorizations.  In the example, we showed factorizations corresponding to both the circle and the ellipse drawn in the figure. }
  \label{fig:diff psd factorizations}
\end{figure}

In every example of a psd factorization that has been presented so far, either the matrices corresponding to the rows or the matrices corresponding to the columns can be chosen to be rank one matrices.  Initial attempts to construct a matrix without this property proved fruitless.  With the machinery of this section, finding such an example becomes almost trivial.

\begin{example}\label{ex:full rank psd factors}
In this example, we present a $4 \times 4$ matrix with psd rank two such that every $2 \times 2$ psd factorization must have a rank two matrix on a row and a rank two matrix on a column.  To construct this example, we start with the $3 \times 3$ derangement matrix as in Example~\ref{ex:orbit of factorization}, which corresponds to the picture shown in Figure~\ref{fig:derangement space}.   Now we add an extra vertex to the inner triangle and an extra facet to the outer triangle so that neither the new vertex nor the new facet touch the circle, as shown in Figure~\ref{fig:augmented derangement}.  This corresponds to a new matrix
\[ M = \left[ \begin{array}{cccc}
0 & 1 & 1 & 2 \\ 1 & 0 & 1 & 2 \\ 1 & 1 & 0 & 6 \\ 1 & 1 & 3 & 3 \end{array} \right]
=
\left[ \begin{array}{ccc}
0 & 1 & 1 \\ 1 & 0 & 1 \\ 1 & 1 & 0 \\ 1 & 1 & 3 \end{array} \right]
\left[ \begin{array}{cccc}
1 & 0 & 0 & 3 \\ 0 & 1 & 0 & 3 \\ 0 & 0 & 1 & -1 \end{array} \right].
\]

The same circle as before is still the unique ellipse nested between the two polytopes so the space of factorizations consists of a single orbit.  When we construct a psd factorization in this orbit, the matrix corresponding to the new vertex must lie in the interior of $\PSD^2$ and the matrix corresponding to the new facet must lie in the interior of $(\PSD^2)^*$.  Hence, they must have rank two.  Such a factorization may be obtained by augmenting our previous factorization for the derangement matrix with the matrix 
$\left[ \begin{array}{cc} 1 & \frac{1}{2} \\ \frac{1}{2} & 1 \end{array} \right]$
for the new vertex and the matrix
$\left[ \begin{array}{cc} 2 & -1 \\ -1 & 2 \end{array} \right]$
for the new facet.
\end{example}

\begin{figure}[ht]
  \includegraphics[width=4.5cm]{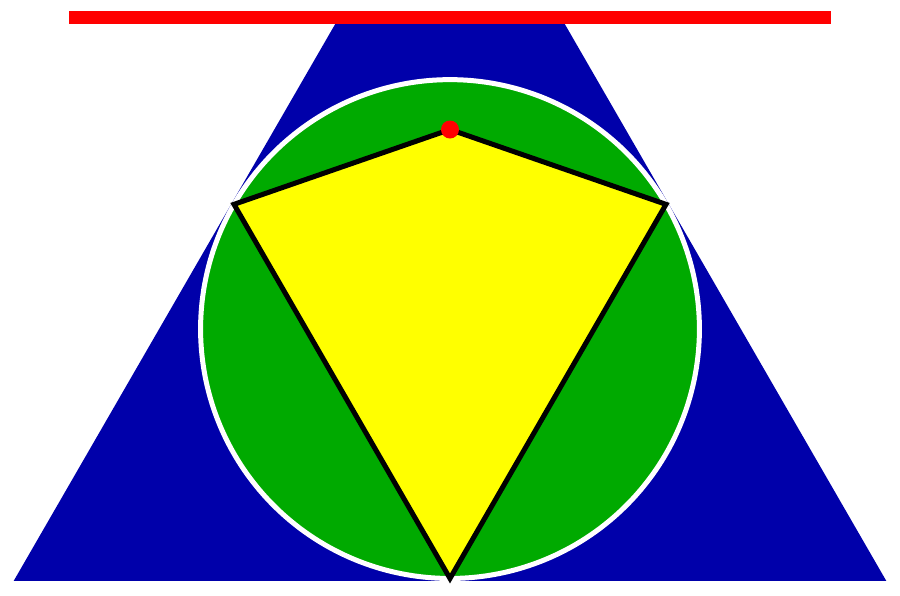}
  \caption{Here we take the picture corresponding to the $3 \times 3$ derangement matrix and add an additional facet to the outer polytope and an additional vertex to the inner polytope (shown in red).  Since the new facet and vertex are not incident to the boundary of the linear embedding of $\PSD^2$, their corresponding matrices in the psd factorization must have full rank.}
  \label{fig:augmented derangement}
\end{figure}

For the special case where $M$ has rank three and we are considering $2 \times 2$ psd factorizations, we can show that the space of factorization orbits is connected.

\begin{proposition} \label{prop:32 connected}
Let $M$ be a nonnegative $p \times q$ matrix with psd rank two and usual rank three.  Then $\SF(M) / GL(2)$ is connected.
\end{proposition}

\begin{proof}
Let $P$ and $Q$ be the cones in $\RR^3$ arising from a rank factorization as above.  By Corollary~\ref{cor:space of factorizations}, it is enough to show that $\Delta_2(P,Q) / GL(2)$ is connected.  To do this, we will dehomogenize the cones so that we can work with polytopes and ellipses.

The cone $Q$ must be pointed since it was formed from a full-rank matrix.  Thus, we can find an affine hyperplane such that the dehomogenization of $Q$ through this hyperplane is bounded.  We dehomogenize through this hyperplane to get polytopes $\widetilde{P} \subset \widetilde{Q}$.  Any element of $\Delta_2(P,Q) / GL(2)$ corresponds to an ellipse nested between $\widetilde{P}$ and $\widetilde{Q}$.  Thus, to finish the proof, it is enough to show that any two ellipses that are nested between $\widetilde{P}$ and $\widetilde{Q}$ can be connected by a path of ellipses.

Suppose $E_0$ and $E_1$ are ellipses nested between the two polytopes.  Then there exist quadratic polynomials $q_0$ and $q_1$ such that $E_i = \left\{ x \; | \; q_i(x) \geq 0 \right\}$.  For $t \in [0,1]$, define a quadratic polynomial $q_t = (1-t)q_0 + t q_1$ and the corresponding ellipse $E_t$.  Since $q_0$ and $q_1$ are nonnegative on the points of $\widetilde{P}$, so is $q_t$ and we have that $\widetilde{P} \subset E_t$.  Similarly, since $q_0$ and $q_1$ are negative on $\left( E_0 \cup E_1 \right)^c$, we have that $E_t \subset E_0 \cup E_1 \subset \widetilde{Q}$.  Thus, $E_t$ gives the desired path of ellipses.
%
\end{proof}

We are not sure if Proposition~\ref{prop:32 connected} extends to matrices $M$ with $\rankpsd (M) = k$ and $\rank (M) = \binom{k+1}{2}$ for $k > 2$.
The proof in the $k=2$ case relied on the fact that bounded spectrahedra in $\PSD^2$ can be represented by a single polynomial inequality.  Higher dimensional spectrahedra require several polynomial inequalities and it is not clear if the proof can be extended to this setting.  Searching for a counterexample has also been difficult, since the next case involves linear images of $\PSD^3$ nested between six-dimensional cones.
\section{Symmetric factorizations}
\label{sec:symmetric_psd_factorizations}

In this section we consider nonnegative matrices that admit \emph{symmetric psd factorizations} where the row and column factors $A_i$ and $B_j$ are required to be the same. Matrices that admit such a factorization are called \emph{completely psd} \cite{laurent2013conic}, by analogy to \emph{completely positive matrices} \cite{berman2003completely}:
\begin{definition}
A symmetric matrix $M \in \S^n$ is called \emph{completely psd} if there exists $k \in \NN$ and $A_1,\dots,A_n \in \S^k_+$ such that $M_{ij} = \langle A_i, A_j \rangle$ for all $i,j=1,\dots,n$.
\end{definition}
The set of matrices that are \emph{completely psd} forms a convex cone in $\S^n$; we denote this cone by $\CS^n$. Completely psd matrices find applications in quantum information theory for the computation of certain quantum graph parameters \cite{laurent2013conic}.

Recall that a matrix $M$ is called \emph{completely positive} if there exist \emph{nonnegative} vectors $a_i$ such that $M_{ij} = \langle a_i, a_j \rangle$ for all $i,j=1,\dots,n$. The convex cone of $n\times n$ completely positive matrices is denoted by $\CP^n$. By representing a nonnegative vector as a diagonal psd matrix it is easy to see that any completely positive matrix is also completely psd, i.e., $\CP^n \subseteq \CS^n$. It is also clear from the definition that any completely psd matrix $M$ is necessarily nonnegative and positive semidefinite. Thus we have the inclusion 
\begin{equation}
\label{eq:inclusionsCPSD}
\CP^n \subseteq \CS^n \subseteq \DN^n
\end{equation}
where $\DN^n$ is the cone of matrices that are nonnegative and positive semidefinite (also called \emph{doubly nonnegative matrices}).

When $n \leq 4$ it is known that $\CP^n = \DN^n$ and thus the inclusions \eqref{eq:inclusionsCPSD} are all equalities.
It is known that for $n = 5$ the two inclusions \eqref{eq:inclusionsCPSD} are strict \cite{laurent2013conic,frenkel2010vector}:
\begin{itemize}
\item To show that $\CP^5 \neq \CS^5$, one can consider the $5\times 5$ matrix $M$ defined by:
\[ M_{ij} = \cos^2\left( \frac{4\pi}{5} (i-j) \right), \quad i,j=1,\dots,5. \]
The matrix $M$ is completely psd since it admits the factorization $M_{ij} = \langle a_i a_i^T, a_j a_j^T \rangle$ where $a_i = (\cos(4\pi i/5), \sin(4\pi i/5)) \in \RR^2$. On the other hand $M \notin \CP^5$ since $\langle H, M \rangle < 0$ where $H$ is an element of the dual cone $(\CP^5)^*$ known as the \emph{Horn form} (see e.g., \cite{dur2010copositive}):
\[
H =
\left[
\begin{array}{rrrrr}
     1  & -1  &  1 &   1 &  -1\\
    -1  &  1  & -1  &  1 &   1\\
     1  & -1  &  1 &  -1 &   1\\
     1  &  1  & -1 &   1 &  -1\\
    -1  &  1  &  1 &  -1 &   1
\end{array}\right].
\]
\item In \cite{laurent2013conic,frenkel2010vector} it was shown that any matrix $M \in \DN^5$ such that $\rank M = 3$ and whose sparsity pattern is the $5$-cycle is \emph{not} completely psd. One can easily exhibit such a matrix; thus this shows that $\CS^5 \neq \DN^5$ (cf. \cite[Eq. 3.3]{laurent2013conic} or \cite[Section 3]{frenkel2010vector} for concrete examples).
\end{itemize}

Several fundamental questions are open concerning the cone $\CS^n$. One important question is to know whether the cone $\CS^n$ is closed. For a completely-psd matrix $M$ one can define the \emph{cpsd-rank} of $M$ as the smallest integer $k$ for which $M$ admits a completely-psd factorization of size $k$. The closedness question concerning $\CS^n$ is related to the question of whether the cpsd-rank of a matrix $M \in \CS^n$ can be bounded from above by a function that depends only on $n$.

%

\section{Open questions}
\label{sec:open}
\subsection{Psd rank of special matrices}

There are very few matrices for which one can determine the psd rank precisely, and when we can, it is usually in very special cases where our coarse bounds such as the square root rank or the trivial rank bound turn out to be sufficient. As such, any new insight on determining the psd rank of concrete matrices will provide new tools to analyze psd rank in general. In that spirit, we propose a few more or less concrete  matrices whose psd ranks we would like to know, and would provide a starting point for this program.

\begin{problem}
Consider the $10$ by $10$ matrix $A$ whose rows and columns are indexed by subsets of $\{1, \ldots, 5\}$ of size $2$ and $3$ respectively, and defined by $A_{I,J}=|I \cap J|$. What is 
the psd rank of $A$?
\end{problem}

One geometric interpretation for the matrix $A$, is to take the
$10$-vertices of a {\em rectified $5$-cell} inscribed in a $3$-sphere and
take the generalized slack matrix with respect to the tangents at
the $10$-points. Since the unit ball in $\RR^4$ has a semidefinite
representation of size $4$, we know that the psd rank of $A$ is at
most $4$. The usual rank of $A$ being $5$, we know that its psd rank
must be at least $3$. If one can show that it is $4$ it would prove
that there is no smaller semidefinite representation of the
$3$-sphere.  A more general version of this problem can be attained by
allowing different set sizes and different numbers of elements.

\begin{problem}
Let $A(n)$ be the matrix whose rows and columns are indexed by subsets of $\{1, \ldots, n\}$ of size $\lfloor n/2 \rfloor$ and $\lceil n/2 \rceil$ respectively, and defined by $A_{I,J}=|I \cap J|$. What is the psd rank of $A(n)$?
\end{problem}

The matrices $A(n)$ again have an interpretation in terms of an inscribed polytope in the $(n-2)$-sphere. In general, we know only that their psd rank is at most $2\lceil\sqrt{n}\rceil$ and at least $\frac{1}{2}\sqrt{1+8n}-\frac{1}{2}$. These matrices are very special, in the sense that they have a rich symmetry structure. In fact they are related to the Johnson scheme $J(n,k)$, with $k=\lfloor n/2 \rfloor$, since if $A_i, i =0,\ldots,k$ are the generators of the associated algebra, $A(n)=\sum_{i=0}^k i A_i$. This justifies posing a more ambitious and
less well-defined question.

\begin{problem}
Let $A$ be a matrix in the Bose-Mesner algebra of the Johnson scheme
(or any other association scheme). Can one exploit the symmetry in
these matrices to provide non-trivial bounds on their psd ranks?
\end{problem}


\subsection{Geometry of psd rank}


In Section~\ref{sec:space of factorizations}, we defined $\SF(M)$, the space of psd factorizations of size $k$ of a nonnegative matrix $M$ of rank ${k+1 \choose 2}$ and psd rank $k$. We showed that the quotient space $\SF(M)/GL(k)$ is connected when $k=2$. 
There is no reason to believe that such a connectivity result holds in general.

\begin{problem} 
\label{prob:space disconnected}
If $M$ is a nonnegative matrix with $\rankpsd (M) = k$ and $\rank (M) = \binom{k + 1}{2}$, then is the space of factorization orbits $\SF(M) / GL(k)$ always connected?  
\end{problem}

The analogous question for nonnegative rank was studied in \cite{straten2003sandwich}.  The authors showed examples where $M$ has rank three and the space of nonnegative factorization orbits of size three is disconnected (here the orbit is generated by the permutation group $S_3$, acting by permuting the entries of the nonnegative factorization).  A natural question to ask is whether these disconnected factorizations remain disconnected when we embed them in the space of $3 \times 3$ psd factorization orbits.  Embedding these factorizations is straightforward (just make the nonnegative vectors into diagonal matrices), but doing so changes the setting of Problem~\ref{prob:space disconnected} 
since we are now considering the possible $3 \times 3$ psd factorizations of a matrix $M$ with rank three and psd rank two.  We are able to show that the factorizations that were disconnected in the nonnegative case become connected when we embed them in the set of $3 \times 3$ psd factorizations. However, this does not settle Problem~\ref{prob:space disconnected}.

\subsection{Complexity and algorithms} Several complexity questions concerning psd rank are open. 

\begin{problem}
For a fixed constant $k$, is it NP-hard to decide if $\rankpsd (M) \leq k$? For example, is it NP-hard to decide if $\rankpsd(M)=3$?
\end{problem}

The equivalent question for the nonnegative rank was shown to be decidable in polynomial time \cite{arora2012computing}. The complexity of the algorithm proposed in \cite{arora2012computing} is polynomial in $p,q$ (the dimensions of the matrix) but doubly-exponential in $k$; this was later improved to a singly-exponential algorithm in $k$ in \cite{moitra2013almost}. Both algorithms express the problem of finding a nonnegative factorization of size $k$ as a semialgebraic system which is then solved using quantifier elimination algorithms \cite{renegar1992intro}. 
Since the complexity of quantifier elimination algorithms has an exponential dependence on the number of variables, one has to make sure that the number of variables in the semialgebraic system is independent of $p,q$ (the size of the nonnegative matrix) in order to get an algorithm that is polynomial in $p,q$. The semialgebraic formulations proposed in \cite{moitra2013almost} and \cite{arora2012computing} rely on the key fact that the solution of any linear program is a rational function of the data. This fact however is far from being true in semidefinite programming \cite{nie2010algebraic}, and this is one major obstacle in extending these algorithms to the psd rank case. 

Another important question is to know the complexity of computing the psd rank (here $k$ is not a fixed constant anymore). In Section \ref{sec:definitions} we saw that the psd rank of a matrix $M \in \RR^{p\times q}_+$ satisfies the following inequalities:
\begin{equation}
\label{eq:ineqbinom}
\rank(M) \leq \binom{\rankpsd(M)+1}{2} \quad \text{ and } \quad \rankpsd(M) \leq \min(p,q)
\end{equation}
It is already not known whether deciding if any of these inequalities is tight can be achieved in polynomial-time.

\begin{problem}
Show that the psd rank is NP-hard to compute. More specifically show that the problems of deciding whether inequalities \eqref{eq:ineqbinom} are tight are NP-hard.
\end{problem}

Note that in \cite[Theorem 4.6]{gouveia2013worst}, an algorithm is proposed to decide whether the first inequality of \eqref{eq:ineqbinom} is tight, however the complexity of the algorithm has an exponential dependence on $\rank(M)$ (the complexity of the algorithm is polynomial when $\rank(M)$ is a fixed parameter).
For the case of the nonnegative rank, it was shown \cite{vavasis2009complexity} that the problem of deciding whether $\rank_+(M) = \rank(M)$ is NP-hard.


%

Another interesting question is to find efficient algorithms to approximate the psd rank.

\begin{problem} Is there a polynomial time approximation algorithm for psd rank (or nonnegative rank) that will find a factorization of size at most $\alpha \cdot \rankpsd$ (or $\alpha \cdot \rankplus$) for some constant $\alpha$?
\end{problem}


\subsection{Slack matrices} The psd rank of slack matrices of polytopes have been of special interest due to its applications to semidefinite lifts of polytopes as described in Section~\ref{sec:motivation}. For most families of polytopes, the growth rate of the psd rank of their slack matrices is unknown.  The tight results that are known are for families where the psd rank is small and grows on the same order as the dimension of the polytopes. A $0/1$ polytope is one whose vertices are $0/1$ vectors. It was shown in \cite{briet2013} that as $n$ grows, most $0/1$-polytopes of dimension $n$ must have psd rank exponential in $n$. The proof works via a counting argument and does not identify specific families of polytopes with such exponential psd rank. 

\begin{problem}
Find an explicit family of $0/1$-polytopes, $\{P_n \subset \RR^n \}$, such that the psd rank of a 
slack matrix of $P_n$ is exponential in $n$.
\end{problem}

Natural candidates for such examples are polytopes that come from NP-hard combinatorial optimization problems such as {\em cut polytopes} and {\em TSP polytopes}.  

It is unknown whether there exists a family of polytopes that can be expressed by liftings to small psd cones but require large polyhedral lifts.  In the language of nonnegative and psd rank, this leads to the following question. 

\begin{problem} \label{prob:large gap}
Find a family of polytopes that exhibits a large (e.g. exponential) gap between 
its psd and nonnegative ranks. 
\end{problem}

There are several families of polytopes with exponentially many facets (in the dimension of the polytope) that can be expressed by small polyhedral lifts or small psd lifts. The stable set polytope of a perfect graph on $n$ vertices has psd rank $n+1$ \cite[Theorem 4.12]{gouveia2013polytopes}. On the other hand, Yannakakis \cite[Theorem 5]{Yannakakis} proved that its nonnegative rank is at most $O(n^{\textup{log} \,n})$. Is nonnegative rank of the stable set polytope of a perfect graph polynomial in $n$ ? Even if it is not, the gap between psd rank and nonnegative rank would not be dramatic for this family. Polytopes with a large gap between psd rank and nonnegative rank as in Problem~\ref{prob:large gap} would show that semidefinite programming is truly more powerful than linear programming in expressing linear optimization problems over polytopes.

Goemans observed in \cite{goemans} that the nonnegative rank of a slack matrix is bounded below by $\log(v)$ where $v$ is the number of vertices of the polytope. The same argument shows that $\log(\sum f_i)$ is a lower bound to nonnegative rank of a polytope where $f_i$ is the number of $i$-dimensional faces of the polytope.  It would be interesting to know if there are similar lower bounds for the psd rank of a polytope coming from the combinatorial structure of the polytope. 

\begin{problem} 
\label{prob:faces and psd rank}
Develop good bounds on the psd rank of a polytope that use information about its combinatorial/facial structure.
\end{problem}

We remark that Corollary~\ref{cor:max facets with fixed psd rank} provides a lower bound to the psd rank of a polytope in terms of the number of facets of the polytope. However, the unknown constants in the bound prevent us from using it to understand specific polytopes. A result in the spirit of Problem~\ref{prob:faces and psd rank} from \cite{gouveia2013worst} is that generic polytopes of dimension $n$ with $v$ vertices have psd rank at least $(nv)^{1/4}$.

\subsection{Completely positive semidefinite matrices}

The following open question was raised in Section \ref{sec:symmetric_psd_factorizations} concerning symmetric psd factorizations.

\begin{problem}
Let $\CS^n$ be the set of $n\times n$ matrices $M$ that admit a factorization of the form \begin{equation}
\label{eq:cpsdfact} M_{ij} = \langle A_i, A_j \rangle
\end{equation}
 where $A_1,\dots,A_n$ are psd matrices. Does there exist a function $f(n)$ such that any matrix $M \in \CS^n$ admits a factorization of the form \eqref{eq:cpsdfact} where the factors $A_1,\dots,A_n$ have size at most $f(n)$?
\end{problem}

\noindent {\bf Acknowledgments}: We thank Troy Lee for sharing his results on the psd rank of Kronecker products as well on the Hermitian psd rank.  The authors also thank Thomas
Rothvo{\ss} for his helpful input in the proof of
Corollary~\ref{cor:eigbound}, and his comments on an earlier draft.

\bibliography{psdrank}

\newcommand{\etalchar}[1]{$^{#1}$}
\begin{thebibliography}{LWdW14}

\bibitem[AGKM12]{arora2012computing}
S.~Arora, R.~Ge, R.~Kannan, and A.~Moitra.
\newblock Computing a nonnegative matrix factorization--provably.
\newblock In {\em Proceedings of the 44th Symposium on Theory of Computing
  (STOC)}, pages 145--162. ACM, 2012.

\bibitem[Bar01]{barvinok2001remark}
A.~Barvinok.
\newblock A remark on the rank of positive semidefinite matrices subject to
  affine constraints.
\newblock {\em Discrete \& Computational Geometry}, 25(1):23--31, 2001.

\bibitem[Bar12]{barvinok2012approximations}
A.~Barvinok.
\newblock Approximations of convex bodies by polytopes and by projections of
  spectrahedra.
\newblock arXiv preprint arXiv:1204:0471, 2012.

\bibitem[BCR11]{bocci2011perturbation}
C.~Bocci, E.~Carlini, and F.~Rapallo.
\newblock Perturbation of matrices and nonnegative rank with a view toward
  statistical models.
\newblock {\em SIAM Journal on Matrix Analysis and Applications},
  32(4):1500--1512, 2011.

\bibitem[BDP13]{briet2013}
J.~Bri\"et, D.~Dadush, and S.~Pokutta.
\newblock On the existence of 0/1 polytopes with high semidefinite extension
  complexity.
\newblock In {\em Algorithms, ESA 2013}, volume 8125 of {\em Lecture Notes in
  Computer Science}, pages 217--228. Springer Berlin Heidelberg, 2013.

\bibitem[BPA{\etalchar{+}}08]{brunner2008testing}
N.~Brunner, S.~Pironio, A.~Acin, N.~Gisin, A.A. M{\'e}thot, and V.~Scarani.
\newblock Testing the dimension of {H}ilbert spaces.
\newblock {\em Physical Review Letters}, 100(21):210503, 2008.

\bibitem[BPT12]{SDOCAG}
G.~Blekherman, P.~A. Parrilo, and R.~Thomas, editors.
\newblock {\em Semidefinite Optimization and Convex Algebraic Geometry},
  volume~13 of {\em MOS-SIAM Series on Optimization}.
\newblock SIAM, 2012.

\bibitem[BSM03]{berman2003completely}
A.~Berman and N.~Shaked-Monderer.
\newblock {\em Completely Positive Matrices}.
\newblock World Scientific Pub Co Inc, 2003.

\bibitem[BV04]{boyd2004convex}
S.P. Boyd and L.~Vandenberghe.
\newblock {\em Convex Optimization}.
\newblock Cambridge University Press, 2004.

\bibitem[CR93]{cohen1993nonnegative}
J.E. Cohen and U.G. Rothblum.
\newblock Nonnegative ranks, decompositions, and factorizations of nonnegative
  matrices.
\newblock {\em Linear Algebra and its Applications}, 190:149--168, 1993.

\bibitem[D{\"u}r10]{dur2010copositive}
M.~D{\"u}r.
\newblock Copositive programming--a survey.
\newblock {\em Recent Advances in Optimization and its Applications in
  Engineering}, pages 3--20, 2010.

\bibitem[FKPT13]{fiorini2013combbounds}
S.~Fiorini, V.~Kaibel, K.~Pashkovich, and D.~O. Theis.
\newblock Combinatorial bounds on nonnegative rank and extended formulations.
\newblock {\em Discrete Mathematics}, 313(1):67 -- 83, 2013.

\bibitem[FMP{\etalchar{+}}12]{fiorini2012lowerbound}
S.~Fiorini, S.~Massar, S.~Pokutta, H.R. Tiwary, and R.~de~Wolf.
\newblock Linear vs. semidefinite extended formulations: Exponential separation
  and strong lower bounds.
\newblock In {\em Proceedings of the Forty-fourth Annual ACM Symposium on
  Theory of Computing}, STOC '12, pages 95--106. ACM, 2012.

\bibitem[FW10]{frenkel2010vector}
P.E. Frenkel and M.~Weiner.
\newblock On vector configurations that can be realized in the cone of positive
  matrices.
\newblock {\em arXiv preprint arXiv:1004.0686}, 2010.

\bibitem[GFR14]{gouveia2014rational}
J.~Gouveia, H.~Fawzi, and R.~Z. Robinson.
\newblock Rational and real positive semidefinite rank can be different.
\newblock {\em arXiv preprint arXiv:1404.4864}, 2014.

\bibitem[GG12]{gillis2012geometric}
N.~Gillis and F.~Glineur.
\newblock On the geometric interpretation of the nonnegative rank.
\newblock {\em Linear Algebra and its Applications}, 437(11):2685--2712, 2012.

\bibitem[GJ79]{GareyJohnson}
M.~R. Garey and D.~S. Johnson.
\newblock {\em Computers and Intractability: A guide to the theory of
  {NP}-completeness}.
\newblock W. H. Freeman and Company, 1979.

\bibitem[Goe14]{goemans}
M.~Goemans.
\newblock Smallest compact formulation for the permutahedron.
\newblock {\em Mathematical Programming Series B}, pages 1--7, 2014.

\bibitem[GPT13]{gouveia2011lifts}
J.~Gouveia, P.A. Parrilo, and R.R. Thomas.
\newblock Lifts of convex sets and cone factorizations.
\newblock {\em Mathematics of Operations Research}, 38(2):248--264, 2013.

\bibitem[GRT13a]{gouveia2013polytopes}
J.~Gouveia, R.~Z. Robinson, and R.~R. Thomas.
\newblock Polytopes of minimum positive semidefinite rank.
\newblock {\em Discrete \& Computational Geometry}, 50(3):679--699, 2013.

\bibitem[GRT13b]{gouveia2013worst}
J.~Gouveia, R.~Z. Robinson, and R.~R. Thomas.
\newblock Worst-case results for positive semidefinite rank.
\newblock {\em arXiv preprint arXiv:1305.4600}, 2013.

\bibitem[Hru12]{hrubes2012edm}
P.~Hrube\v{s}.
\newblock On the nonnegative rank of distance matrices.
\newblock {\em Information Processing Letters}, 112:457--461, 2012.

\bibitem[Jol02]{JolliffePCA}
I.~Jolliffe.
\newblock {\em Principal component analysis}.
\newblock Springer Series in Statistics. Springer, 2nd edition, 2002.

\bibitem[JSWZ13]{jain2013efficient}
R.~Jain, Y.~Shi, Z.~Wei, and S.~Zhang.
\newblock Efficient protocols for generating bipartite classical distributions
  and quantum states.
\newblock {\em IEEE Transactions on Information Theory}, 59(8):5171--5178,
  2013.

\bibitem[Kal63]{kalman1963mathematical}
R.~E. Kalman.
\newblock Mathematical description of linear dynamical systems.
\newblock {\em Journal of the Society for Industrial \& Applied Mathematics,
  Series A: Control}, 1(2):152--192, 1963.

\bibitem[LMSS07]{LinialMendelsonSchechtmanShraibman}
N.~Linial, S.~Mendelson, G.~Schechtman, and A.~Shraibman.
\newblock Complexity measures of sign matrices.
\newblock {\em Combinatorica}, 27(4):439--463, 2007.

\bibitem[LP13]{laurent2013conic}
M.~Laurent and T.~Piovesan.
\newblock Conic approach to quantum graph parameters using linear optimization
  over the completely positive semidefinite cone.
\newblock {\em arXiv preprint arXiv:1312.6643}, 2013.

\bibitem[LS99]{LeeSeung}
D.~D. Lee and H.~S. Seung.
\newblock Learning the parts of objects by non-negative matrix factorization.
\newblock {\em Nature}, 401(6755):788--791, 1999.

\bibitem[LS09]{LinialShraibman}
N.~Linial and A.~Shraibman.
\newblock Lower bounds in communication complexity based on factorization
  norms.
\newblock {\em Random Structures \& Algorithms}, 34(3):368--394, 2009.

\bibitem[LT12]{leetheis2012support}
T.~Lee and D.~O. Theis.
\newblock Support-based lower bounds for the positive semidefinite rank of a
  nonnegative matrix.
\newblock arXiv preprint arXiv:1203.3961, 2012.

\bibitem[LWdW14]{TroyLeeHermitianPsdRank}
T.~Lee, Z.~Wei, and R.~de~Wolf.
\newblock Some upper and lower bounds on psd rank.
\newblock In preparation, 2014.

\bibitem[Moi13]{moitra2013almost}
A.~Moitra.
\newblock An almost optimal algorithm for computing nonnegative rank.
\newblock In {\em SODA}, pages 1454--1464. SIAM, 2013.

\bibitem[Moo81]{MoorePCA}
B.~Moore.
\newblock Principal component analysis in linear systems: Controllability,
  observability, and model reduction.
\newblock {\em IEEE Transactions on Automatic Control}, 26(1):17--32, 1981.

\bibitem[MSvS03]{straten2003sandwich}
D.~Mond, J.~Smith, and D.~van Straten.
\newblock Stochastic factorizations, sandwiched simplices and the topology of
  the space of explanations.
\newblock {\em Royal Soc. Proc., Math. Phys. and Engineering Sciences},
  459(2039):2821--2845, 2003.

\bibitem[NRS10]{nie2010algebraic}
J.~Nie, K.~Ranestad, and B.~Sturmfels.
\newblock The algebraic degree of semidefinite programming.
\newblock {\em Mathematical Programming}, 122(2):379--405, 2010.

\bibitem[Pat98]{pataki1998rank}
G.~Pataki.
\newblock On the rank of extreme matrices in semidefinite programs and the
  multiplicity of optimal eigenvalues.
\newblock {\em Mathematics of operations research}, 23(2):339--358, 1998.

\bibitem[Ren92]{renegar1992intro}
J.~Renegar.
\newblock On the computational complexity and geometry of the first-order
  theory of the reals. {P}art {I}: Introduction. preliminaries. {T}he geometry
  of semi-algebraic sets. {T}he decision problem for the existential theory of
  the reals.
\newblock {\em Journal of Symbolic Computation}, 13(3):255 -- 299, 1992.

\bibitem[Ren06]{renegar2006hyperbolic}
J.~Renegar.
\newblock Hyperbolic programs, and their derivative relaxations.
\newblock {\em Foundations of Computational Mathematics}, 6(1):59--79, 2006.

\bibitem[Sch12]{scheiderer2012representations}
C.~Scheiderer.
\newblock Semidefinite representation for convex hulls of real algebraic
  curves.
\newblock arXiv preprint arXiv:1208.3865, 2012.

\bibitem[Vav09]{vavasis2009complexity}
S.~A. Vavasis.
\newblock On the complexity of nonnegative matrix factorization.
\newblock {\em SIAM Journal on Optimization}, 20(3):1364--1377, 2009.

\bibitem[VB96]{vandenberghe1996semidefinite}
L.~Vandenberghe and S.~Boyd.
\newblock Semidefinite programming.
\newblock {\em SIAM Review}, 38(1):49--95, 1996.

\bibitem[WCD08]{wehner2008lower}
S.~Wehner, M.~Christandl, and A.C. Doherty.
\newblock Lower bound on the dimension of a quantum system given measured data.
\newblock {\em Physical Review A}, 78(6):062112, 2008.

\bibitem[Yan91]{Yannakakis}
M.~Yannakakis.
\newblock Expressing combinatorial optimization problems by linear programs.
\newblock {\em J. Comput. System Sci.}, 43(3):441--466, 1991.

\end{thebibliography}
\bibliographystyle{alpha}

\end{document}